\documentclass[a4paper,12pt, reqno]{amsart}
\usepackage{amsthm}
\usepackage{amsmath,amssymb,latexsym,amsfonts,mathrsfs}
\usepackage{color}
\usepackage[dvipdfmx]{graphicx}
\usepackage{bbm}
\usepackage{hyperref}
\usepackage{enumitem}
\usepackage{mathtools} 
\mathtoolsset{showonlyrefs=true}

 \usepackage[foot]{amsaddr}

\renewcommand{\hat}{\widehat}
\renewcommand{\tilde}{\widetilde}

\newcommand{\bN}{\ensuremath{\mathbb{N}}}

\newcommand{\bR}{\ensuremath{\mathbb{R}}}

\newcommand{\cA}{\ensuremath{\mathcal{A}}}
\newcommand{\cB}{\ensuremath{\mathcal{B}}}

\theoremstyle{plain}
\newtheorem{Thm}{Theorem}[section]

\newtheorem{Lem}[Thm]{Lemma}

\newtheorem{Cor}[Thm]{Corollary}

\theoremstyle{definition}
\newtheorem{Asmp}[Thm]{Assumption}

\newtheorem{Rem}[Thm]{Remark}
\newtheorem{Ex}[Thm]{Example}

\numberwithin{equation}{section}

\begin{document}

\title{Large deviations for occupation and waiting times of infinite ergodic transformations}

\author{Toru Sera}

\begin{abstract}
We establish large deviation estimates related to the Darling--Kac theorem and generalized arcsine laws for occupation and waiting times of ergodic transformations preserving an infinite measure, such as non-uniformly expanding interval maps with indifferent fixed points. For the proof, we imitate the study of generalized arcsine laws for occupation times of one-dimensional diffusion processes and adopt a method of double Laplace transform.
\end{abstract}

\address[Toru Sera]{Department of Mathematics, Graduate School of Science, Osaka University,
Toyonaka, Osaka 560-0043, Japan.}
\email{sera@math.sci.osaka-u.ac.jp}

\subjclass{Primary 37A40; Secondary 37A50, 60F10}
\keywords{infinite ergodic theory, large deviation estimates, Darling--Kac theorem, generalized arcsine laws}

\maketitle


\section{Introduction}

In the study of dynamical systems with an infinite invariant measure, a variety of ergodic and probabilistic limit theorems have been established. They are often related to classical limit theorems for renewal, Markov or diffusion processes in probability theory.
Among this kind of research of dynamical systems, we are going to focus on three distributional limit theorems, the Darling--Kac theorem for occupation times of sets of finite measure, 
 the Dynkin--Lamperti generalized arcsine law for the last time the orbit visits to  sets of finite measure, and 
the Lamperti generalized arcsine law for occupation times of sets of infinite measure, studied by \cite{Aa81, Aa86, Tha98, Tha02, ThZw, Zwe07cpt, KocZwe, OwaSam,  SerYan19, Ser20}. The aim of the present paper is to establish large deviation estimates related to these limit theorems under  similar abstract settings as in \cite{ThZw, Zwe07cpt, KocZwe, SerYan19}. Our abstract results can be applied to, for examples, intermittent maps, that is, non-uniformly expanding interval maps with indifferent fixed points.
We are motivated by the study of a large deviation estimate related to a generalized arcsine law for occupation times of one-dimensional diffusion processes \cite{KasYan}.
We also refer the reader to \cite{RYZ} for another type of large deviations, which is related to the strong arcsine law for a one-dimensional Brownian motion.
Let us illustrate earlier study by giving an example as in \cite[Example 1.1]{ThZw}.

\begin{Ex}[Distributional limit theorems for Boole's transformation]\label{Ex:boole}
We refer the reader to \cite{AdlWei, Aa97, Tha02} for the details of Boole's transformation.
The map $T:[0,1]\to[0,1]$ given by
\begin{align}
     Tx=
     \begin{cases}
     x(1-x)/(1-x-x^2), &x\in[0,1/2],
     \\
     1-T(1-x), &x\in(1/2,1],	
     \end{cases}
\end{align}
is conjugated to Boole's transformation $\tilde{T}x=x-x^{-1}$ $(x\in\bR\setminus\{0\})$. Indeed, let $\phi(x)=(1-x)^{-1}-x^{-1}$ $(x\in(0,1))$, then $\tilde{T}=\phi\circ T \circ \phi^{-1}$ on $\bR\setminus \{0\}$.
It is easy to see that $T0=0, T1=1$, $T'(0)=T'(1)=1$, $T''>0$ on $(0,1/2)$ and $T''<0$ on $(1/2,1)$. In addition we have $Tx-x = 1-x -T(1-x)\sim x^3 $ $(x\to0)$.
Thus $T$ is a special case of Thaler's maps, which will be explained in Section \ref{Sec:Thaler}.
The map $T$ admits the invariant density $h$ given by
\begin{align}
    h(x)
    =\frac{1}{x^2}+\frac{1}{(1-x)^2},
    \quad
    x\in(0,1).	
\end{align}
Therefore the invariant measure $\mu$ given by $d\mu(x)=h(x)\,dx$ $(x\in[0,1])$ is an infinite measure.	
Set $\gamma=\sqrt{2}-1\in(0,1/2)$, which is a $2$-periodic point of $T$. Indeed, $T\gamma=1-\gamma\in(1/2,1)$ and hence $T^2\gamma=\gamma$. Let
\begin{align}
   A_0=[0,\gamma),\quad
   Y=[\gamma,T\gamma],\quad
   A_1=(T\gamma,1].	
\end{align}
Then $\mu(Y)=\sqrt{2}$ and $\mu(A_0)=\mu(A_1)=\infty$. In addition, $Y$ dynamically separates $A_0$ and $A_1$, that is, $A_i\cap T^{-1}A_j=\emptyset$ $(i\neq j)$. 
For a non-negative integer $n$, a Borel subset $A\subset [0,1]$ and $x\in[0,1]$, set
\begin{align}
   S_n^A(x)=\sum_{k=1}^{n}1_A(T^k x),
   \quad
   Z_n^A(x)=\max\{k\leq n\::\:T^k x\in A\}.	
\end{align}
In other words, $S_n^A(x)$ denotes the occupation time on $A$ of the orbit $\{T^k x\}_{k\geq0}$ between time $1$ and $n$, and $Z_n^A(x)$ denotes the last time the orbit arrives in $A$ until time $n$. 
Fix any Borel probability measure $\nu(dx)$ absolutely continuous with respect to the Lebesgue measure on $[0,1]$.
We interpret $x$ as the initial point of the orbit $\{T^k x\}_{k\geq0}$, and $\nu(dx)$ as the initial distribution of the orbit.
Then the Darling--Kac theorem \cite{Aa81, Aa86} yields that, as $n\to\infty$,
\begin{align}\label{boole-DK}
	\nu\bigg(\frac{\pi S_n^Y}{2\sqrt{n}}\leq t\bigg)
    \to
	\frac{2}{\pi}\int_0^t e^{-y^2/\pi}\,dy,
	\quad t\geq0.
\end{align}
Next,  the Dynkin--Lamperti generalized arcsine law for waiting time \cite{Tha98} shows that, as $n\to\infty$,
\begin{align}\label{boole-DL}
	\nu\bigg(\frac{Z_n^Y}{n}\leq t\bigg)
	\to
	\frac{2}{\pi}\arcsin\sqrt{t},
	\quad
	t\in[0,1].
\end{align}
And finally, the Lamperti generalized arcsine law for occupation time \cite{Tha02} implies that, as $n\to\infty$,
\begin{align}\label{boole-L}
	\nu\bigg(\frac{S_n^{A_i}}{n}\leq t\bigg)
	\to
	\frac{2}{\pi}\arcsin\sqrt{t},
	\quad
	t\in[0,1],
	\;i=0,1.
\end{align}
We also remark that convergence rates of \eqref{boole-DK} and \eqref{boole-DL} were also studied in \cite{MelTer, Ter14, Ter15, Ter16}, and a large deviation estimate for the Perron--Frobenius operator related to \eqref{boole-DL} can be found in \cite{Tha05}.
\end{Ex}

We now illustrate our main results. Our aim is to estimate
the left-hand sides of \eqref{boole-DK}, \eqref{boole-DL} and \eqref{boole-L} with $t\to0$.

\begin{Ex}[Large deviation estimates for Boole's transformation]\label{Ex:boole-2}
Under the setting of Example \ref{Ex:boole}, 
we further assume that $\nu$ is a probability measure supported on $[\varepsilon, 1-\varepsilon]$ for some $\varepsilon\in(0,1/2)$, and admits a Riemann integrable density. Then there exists some constants $0<C_1\leq C_2<\infty$ such that, for any positive-valued sequence $\{c(n)\}_{n\geq0}$ with $c(n)\to0$ and $c(n)n\to\infty$ $(n\to\infty)$, the following three estimates hold: 
\begin{align}
	C_1
	&\leq \liminf_{n\to\infty}\frac{1}{\sqrt{c(n)}}\nu\bigg(\frac{ S_n^Y }{\sqrt{n}}\leq \sqrt{c(n)}\bigg)
	\\
	&\leq
	\limsup_{n\to\infty}\frac{1}{\sqrt{c(n)}}\nu\bigg(\frac{ S_n^Y}{\sqrt{n}}\leq \sqrt{c(n)}\bigg)
	\leq C_2,
	\label{boole-LD-1}
\end{align}
and
\begin{align}
	C_1
	&\leq \liminf_{n\to\infty}\frac{1}{\sqrt{c(n)}}
	\nu\bigg(\frac{Z_n^Y}{n}\leq c(n)\bigg)
	\\
	&\leq \limsup_{n\to\infty}\frac{1}{\sqrt{c(n)}}\nu\bigg(\frac{Z_n^Y}{n}\leq c(n)\bigg)
	\leq C_2,
	\label{boole-LD-2}
\end{align}
and
\begin{align}
C_1
	&\leq \liminf_{n\to\infty}\frac{1}{\sqrt{c(n)}}
	\nu\bigg(\frac{S_n^{A_i}}{n}\leq c(n)\bigg)
	\\
	&\leq \limsup_{n\to\infty}\frac{1}{\sqrt{c(n)}}\nu\bigg(\frac{S_n^{A_i}}{n}\leq c(n)\bigg)
	\leq C_2,
	\quad i=0,1.
	\label{boole-LD-3}	
\end{align}
Note that $C_1$ and $C_2$ may depend on $\nu$.
These estimates seem to be compatible with \eqref{boole-DK}, \eqref{boole-DL} and \eqref{boole-L} respectively, since
the right-hand side of \eqref{boole-DK} with $t=\sqrt{c(n)}$ and those of \eqref{boole-DL} and \eqref{boole-L} with $t=c(n)$ are asymptotically equal to $(2/\pi)\sqrt{c(n)}$, as $n\to\infty$. Nevertheless, \eqref{boole-DK}, \eqref{boole-DL} and \eqref{boole-L} do not imply \eqref{boole-LD-1}, \eqref{boole-LD-2} and \eqref{boole-LD-3}  directly.

\end{Ex}

For the proof, we adopt a method of double Laplace transform as in \cite{SerYan19}, imitating the study of generalized arcsine laws for occupation times of one-dimensional diffusion processes \cite{BPY, Wat95, RYZ, KasYan, Ya.Y17}. Although moment methods were used in \cite{Tha98, Tha02, ThZw, Zwe07cpt, KocZwe}, double Laplace transform is more adequate for our large deviation estimates. 
For example, the probability $\nu(Z_n^Y/n\leq c(n))$ in Example \ref{Ex:boole-2} has a negligibly small contribution to the $k$th moment $\int (Z_n^Y/c(n)n)^k\,d\nu$ $(k=1,2,\dots)$, while it has large contributions to Laplace transform
\begin{align}
 \int \exp\bigg(-\frac{\lambda Z_n^Y}{c(n)n}\bigg)\,d\nu
 \quad(\lambda>0)
 \end{align}
 and double Laplace transform
 \begin{align}
 \int_0^\infty du \,e^{-qu}\int d\nu\, 
 \exp\bigg(-\frac{\lambda Z_{[un]}^Y}{c(n)n}\bigg)\quad (q,\lambda>0).
 \end{align} 
 This is why we adopt a method of double Laplace transform rather than moment methods in order to estimate $\nu(Z_n^Y/n\leq c(n))$.

This paper is organized as follows.
 In Section \ref{Sec:Pre} we recall some basic notions of infinite ergodic theory and the theory of regular variation. 
In Section \ref{Sec:Main} we formulate large deviation estimates related to the Darling--Kac theorem and generalized arcsine laws in abstract settings. 
Section \ref{Sec:Ana} is devoted to introduce some lemmas needed to calculate double Laplace transform. 
In Sections \ref{Sec:proof-Z}, \ref{Sec:proof-SY} and \ref{Sec:proof-SA} we prove the large deviation estimates by using double Laplace transform.
In Section \ref{Sec:Thaler} we apply our abstract results to Thaler's maps.

\section{Preliminaries}\label{Sec:Pre}
Before presenting our main results, let us recall basic concepts of infinite ergodic theory.
We basically follow the settings of \cite{ThZw, Zwe07cpt, KocZwe, SerYan19}. We also refer the reader to \cite{Aa97} for the foundations of infinite ergodic theory. 

Let $\bN$ denote the set of all positive integers, and set $\bN_0=\bN\cup\{0\}$.
We always assume $(X,\cA, \mu)$ is a $\sigma$-finite measure space with $\mu(X)=\infty$, and 
a measurable map $T:X \to X$ is a conservative, ergodic, measure-preserving transformation on $(X,\cA,\mu)$, which is abbreviated as \emph{CEMPT}. 
For  $A\in \cA$, we write $1_A$ for the indicator function of $A$. 
Since $T$ is a CEMPT,  we have $\sum_{n\geq 0} 1_A\circ T^{n} =\infty$, a.e.\ for any $A\in\cA$ with $\mu(A)>0$. 
In other words, the orbit $\{T^n x\}_{n\geq0}$ visits $A$ infinitely often for $\mu$-almost every initial point $x$.
For $u\in L^1(\mu)$, define the signed measure $\mu_u$ on $(X,\cA)$ as $\mu_u(A)=\int_A u\,d\mu$ $(A\in \cA)$. 
The \emph{transfer operator} $\hat{T}:L^1(\mu)\to L^1(\mu)$ is 
defined by $\hat{T}u=d(\mu_u\circ T^{-1})/d\mu$ $(u\in L^1(\mu))$. 
This operator is
characterized by the equation $ \int_X (v\circ T) u \,d\mu =
    \int_X v(\hat{T}u)\,d\mu$
for any $v\in L^{\infty}(\mu)$ and $u\in L^1 (\mu)$. 
The domain of $\hat{T}$ can be extended to all  non-negative, measurable functions $u:X\to [0,\infty)$. 
Then $\int_X \hat{T}u \, d\mu=\int_X u \, d\mu$ for any non-negative, measurable function $u$.

We need to extend the concept of uniform sweeping of \cite{Zwe07cpt} slightly.
Let $\mathfrak{H}\cup\{G\}$ be a family of some non-negative, measurable functions $H:X\to [0,\infty)$. 
We say $\mathfrak{H}$ is \emph{uniformly sweeping (in $K$ steps) for $G$} if the following condition holds:  there exist some constants $C>0$ and $K\in\bN_0$ such that, for any $H\in\mathfrak{H}$, we have $C\sum_{k=0}^K \hat{T}^k H\geq G$ a.e.
In addition, if $\mathfrak{H}=\{H\}$, then we say $H$ is \emph{uniformly sweeping (in $K$ steps) for $G$}.

Let us recall regularly and slowly varying functions. We refer the reader to \cite{BGT} for the details. 
Let $f, g: (0,\infty)\to(0,\infty)$ be positive, measurable functions. 
If $f(t)/g(t)\to c\in[0,\infty]$ $(t\to t_0)$, then we write $f(t)\sim c g(t)$ ($t\to t_0$).
 We say $f$ is \emph{regularly varying of index $\rho\in \bR$ at $\infty$} (respectively, \emph{at $0$}) if, for any $\lambda>0$,
\begin{align*}
f(\lambda t) \sim \lambda^\rho f(t) \quad\text{$(t\to\infty)$ (respectively, $t\to 0+$)}. 	
\end{align*}
In the case $\rho=0$, we say $f$ is \emph{slowly varying at $\infty$} (respectively, \emph{at $0$}). 
A positive-valued sequence $\{a(n)\}_{n\geq 0}$ is called regularly varying of index $\rho$ if the function $a( [t])$ is regularly varying of index $\rho$ at $\infty$. Here $[t]$ denotes the greatest integer which is less than or equal to $t$.


Fix a reference set $Y\in \cA$ with $\mu(Y)\in(0,\infty)$ from now on. Let $\varphi:X\to \bN\cup\{\infty\}$ be a \emph{first return time} to $Y$, that is, 
\begin{equation}
\varphi(x)=\min\{k\geq1\::\:T^k x\in Y\}\quad(x\in X).
\end{equation}
Here it is understood that $\min\emptyset=\infty$. 
Define disjoint sets $Y_0, Y_1, Y_2,\ldots\in \cA$ as
\begin{equation}
	Y_0=Y,\quad  Y_n= (T^{-n}Y)\setminus \bigg(\bigcup_{k=0}^{n-1}T^{-k}Y \bigg)=Y^c\cap\{\varphi=n\} \quad (n\in\bN).
\end{equation}
As proved in \cite[(2.3)]{ThZw},  
\begin{equation}
	 \label{eq:1Yn}
	1_{Y_n}
	=
	\sum_{k>n}\hat{T}^{k-n}1_{Y\cap \{\varphi=k\}}
	\quad
	\text{a.e.\;\;($n\in\bN_0$)},
\end{equation}
and $\mu(Y_n)=\mu(Y\cap\{\varphi>n\})$.
Let $\{w_n^Y\}_{n\geq0}$ denote the \emph{wandering rate} of $Y$, which is given by
\begin{equation}
	w_n^Y
	=
	\mu\bigg(\bigcup_{k=0}^{n-1}T^{-k}Y\bigg)
	=
	\sum_{k=0}^{n-1}\mu(Y_k)
   = \sum_{k=0}^{n-1} \mu(Y\cap\{\varphi>k\})
   \quad (n\in\bN_0).
\end{equation}
Since $T$ is a CEMPT, we see $\bigcup_{n\geq0}T^{-n}Y=X$, a.e.\ and hence $w_n^Y\to\infty$  ($n\to\infty$).
For $s>0$, let $Q^Y(s)$  be  Laplace transform of $\{w_{n+1}^Y-w_n^Y\}_{n\geq0}$:
\begin{equation}
	Q^Y(s)=\sum_{n\geq0} e^{-ns}\big(w_{n+1}^Y-w_n^Y\big)
	=\sum_{n\geq0}e^{-ns}\mu(Y\cap\{\varphi>n\})
	\quad
	(s>0).
\end{equation}
Then $0<Q^Y(s)<\infty$ and $Q^Y(s)\to \infty$ ($s\to 0+$). 
Let $\alpha\in(0,1)$ and let $\ell:(0,\infty)\to(0,\infty)$ be a positive, measurable function slowly varying at $\infty$.
By Karamata's Tauberian theorem \cite[Theorem 1.7.1]{BGT}, the condition 
\begin{equation}
	w_n^Y
	\sim
	n^{1-\alpha}\ell(n)
	\quad
	(n\to\infty)
\end{equation}
is equivalent to
\begin{align}\label{eq:RV:Q}
	 Q^Y(s)\sim \Gamma(2-\alpha) s^{-1+\alpha} \ell(s^{-1})
	 \quad
	 (s\to 0+).
\end{align}
Here $\Gamma(z)=\int_0^\infty e^{-t}t^{-1+z}dt$ ($z>0$) denotes the gamma function. 

If $\{(w_n^Y)^{-1}\sum_{k=0}^{n-1}\hat{T}^k 1_{Y_k}\}_{n\geq1}$  converges in $L^\infty(\mu)$ as $n\to\infty$,
then we call the limit function $H\in L^\infty(\mu)$ as the \emph{asymptotic entrance density} of $Y$. Since $(w_n^Y)^{-1}\sum_{k=0}^{n-1}\hat{T}^k 1_{Y_k}$ is a $\mu$-probability density function supported on $Y$, so is $H$. Let $G\in\{u\in L^1(\mu)\::\:u\geq0\}$. Then $H$ is uniformly sweeping in $K$ steps for $G$ if and only if there exists $N\in\bN$ such that $\{(w_n^Y)^{-1}\sum_{k=0}^{n-1}\hat{T}^k 1_{Y_k}\}_{n\geq N}$ is uniformly sweeping in $K$ steps for $G$.

\section{Main results}\label{Sec:Main}

In the following, we are going to formulate three types of large deviation estimates, which are related to already-known distributional limit theorems.

\subsection{Large deviation estimates related to the Dynkin--Lamperti generalized arcsine law}

Let $u:X\to[0,\infty)$ be a non-negative, $\mu$-integrable function. Recall $\mu_u$ is defined as the $\mu$-absolutely continuous finite measure on $X$ with density function $u$ with respect to $\mu$, that is,
\begin{equation}
	\mu_{u}(A)
	=\int_A u(x)\,d\mu(x)
	\quad(A\in\cA).
\end{equation}
Fix a reference set $Y\in \cA$. Let $Z_n^Y(x)$ denote the last time the orbit $\{T^k x\}_{k\geq0}$ arrives in $Y$ until time $n$, that is,
\begin{align}
	Z_n^Y(x) = \max\{k\leq n: T^k x \in Y\}
	\quad (n\in\bN_0,\;x\in X),
\end{align}
where it is understood that $\max \emptyset=0$.


\begin{Thm}\label{Thm:Z}
Suppose the following conditions \ref{Thm:Z-cond-1}, \ref{Thm:Z-cond-2}, \ref{Thm:Z-cond-3}, \ref{Thm:Z-cond-4} are satisfied:
\begin{enumerate}[label=\textup{(A\arabic*)}]
	\item\label{Thm:Z-cond-1} $(X, \cA, \mu)$ is a $\sigma$-finite measure space with $\mu(X)=\infty$, $Y\in \cA$, $\mu(Y)\in(0,\infty)$ and $T$ is a CEMPT on $(X,\cA, \mu)$.

    \item\label{Thm:Z-cond-2} $w_n^Y\sim n^{1-\alpha}\ell(n)$ $(n\to\infty)$, for some $\alpha\in(0,1)$ and some positive, measurable function $\ell:(0,\infty)\to(0,\infty)$ slowly varying at $\infty$. 
	\item\label{Thm:Z-cond-3} There exists $N\in \bN$ such that $\displaystyle \bigg\{\frac{1}{w_n^{Y}}\sum_{k=0}^{n-1}\hat{T}^k 1_{Y_k}\bigg\}_{n\geq N}$ is  uniformly sweeping for $1_Y$.

	\item\label{Thm:Z-cond-4} There exists $H\in L^\infty(\mu)$ such that 
\begin{equation}
	\lim_{n\to\infty}\frac{1}{w_n^Y}
	 \sum_{k=0}^{n-1} \hat{T}^k 1_{Y_k}
   = H
   \quad \text{in $L^\infty(\mu)$}.
\end{equation}

\end{enumerate}
Let $\{c(n)\}_{n\geq0}$ be a positive-valued sequence satisfying 
\begin{equation}\label{eq:c(n)}
	c(n)\to0, \;\; c(n)n\to\infty  \quad(n\to\infty).
\end{equation}
Then
\begin{align}\label{eq:Thm:Z}
  \mu_{H}\bigg(\frac{Z^Y_n}{n} \leq c(n) \bigg)
  \sim
 \frac{\sin(\pi\alpha)}{\pi\alpha} \frac{c(n)^\alpha \ell(n)}{\ell(c(n)n)}
 \quad(n\to\infty).  	
\end{align}
\end{Thm}

The proof of Theorem \ref{Thm:Z} will be given in Section \ref{Sec:proof-Z}.

\begin{Rem}
Under the setting of Theorem \ref{Thm:Z}, 
fix $\varepsilon\in(0,\alpha)$ arbitrarily. 
Then the Potter bounds for slowly varying functions \cite[Theorem 1.5.6]{BGT}
implies that there exist $C_\varepsilon\geq 1$ and $N_\varepsilon\in\bN$ such that, for any $n\geq N_\varepsilon$, we have $c(n)\leq 1$ and
\begin{align}
    	C_\varepsilon^{-1} c(n)^\varepsilon
    	\leq \frac{\ell(n)}{\ell(c(n)n)}\leq C_\varepsilon c(n)^{-\varepsilon}.
\end{align}
Thus the right-hand side of \eqref{eq:Thm:Z} converges to $0$ as $n\to\infty$. 
If we further assume 
\begin{align}\label{eq:c(n)-2}
\frac{\ell(n)}{\ell(c(n)n)}\to 1 \quad(n\to\infty),
\end{align}	
then we obtain 
\begin{align}\label{eq:Thm:Z-modif}
  \mu_{H}\bigg(\frac{Z^Y_n}{n} \leq c(n) \bigg)
  \sim
 \frac{\sin(\pi\alpha)}{\pi\alpha} c(n)^\alpha
 \quad(n\to\infty).  	
\end{align}	
\end{Rem}

\begin{Rem}
Fix any positive, measurable function $\ell:(0,\infty)\to(0,\infty)$ slowly varying at $\infty$. Then there exists a non-increasing, positive-valued sequence $\{c(n)\}_{n\geq0}$ satisfying \eqref{eq:c(n)} and \eqref{eq:c(n)-2}.
Indeed, we use the uniform convergence theorem for slowly varying functions \cite[Theorem 1.2.1]{BGT} to take an strictly increasing sequence $\{M_N\}_{N\geq1}\subset \bN$ so that
\begin{equation}
    \sup \bigg\{\bigg|\frac{\ell(t)}{\ell(\lambda t)}-1\bigg|\::\:\lambda\in[N^{-1}, 1], \;t\geq M_N \bigg\}
    \leq
    \frac{1}{N}
    \quad(N\in\bN).	
\end{equation}
Set $c(n)=1$ for $0\leq n < M_1$ and $c(n)=N^{-1/2}$ for $M_N\leq n < M_{N+1}$ $(N\in\bN)$. It is easy to check that $\{c(n)\}_{n\geq0 }$ satisfies \eqref{eq:c(n)} and \eqref{eq:c(n)-2}.	
\end{Rem}

\begin{Rem}[Comparison to the Dynkin--Lamperti generalized arcsine law]
Let us recall the Dynkin--Lamperti generalized arcsine law for waiting times.
Assume the conditions \ref{Thm:Z-cond-1}, \ref{Thm:Z-cond-2}, \ref{Thm:Z-cond-4} of Theorem \ref{Thm:Z} are fulfilled.
Then for any $\mu$-absolutely continuous probability measure $\nu$ on $(X,\cA)$ and any $0\leq t\leq 1$, we have 
 \begin{align}\label{eq:Arcsine-Z}
     \lim_{n\to\infty}\nu\bigg(\frac{Z^Y_n}{n} \leq t \bigg)
     =
     \frac{\sin(\pi\alpha)}{\pi}\int_0^t \frac{ds}{s^{1-\alpha}(1-s)^{\alpha}}, 
\end{align}
which follows from \cite[Theorem 2.3]{Zwe07cpt}. See also \cite[Theorem 3.3]{ThZw} and \cite[Theorem 2.1]{KocZwe}. The limit is the distribution function of the $\rm{Beta}(\alpha,1-\alpha)$-distribution. In the case $\alpha=1/2$, this distribution is the usual arcsine distribution.
We emphasize that the right-hand side of \eqref{eq:Arcsine-Z} does not depend on the choice of $\nu$ because of the ergodicity of $T$.
Note that the right-hand side of \eqref{eq:Thm:Z-modif}  is asymptotically the same as the right-hand side of \eqref{eq:Arcsine-Z} with $t=c(n)$, as $n\to\infty$.
Nevertheless \eqref{eq:Thm:Z-modif} does not follow from \eqref{eq:Arcsine-Z} directly. 
We do not know whether \eqref{eq:Thm:Z} remains valid in the case $\mu_H$ is replaced by other suitable probability measures $\nu$ except for $\mu_{\hat{T}^k H}$ (see also Corollaries \ref{Cor1:Z}, \ref{Cor2:Z}, Theorem \ref{Thm:Z-weak} and Remark \ref{Rem:limsup=infty}).
The difficulty is that the $L^1$-characterization of the ergodicity \cite[Theorem 3.1]{Zwe07cpt} is inadequate for this purpose, although it is significant for \eqref{eq:Arcsine-Z}. 
\end{Rem}


In the following two corollaries, we will consider what happens when we replace $\mu_{H}$ in the left-hand side of \eqref{eq:Thm:Z} by other finite measures.

\begin{Cor}\label{Cor1:Z}
Let $k\in \bN_{0}$. Under the setting of Theorem $\ref{Thm:Z}$,
\begin{align}
  \mu_{H}\bigg(\frac{Z^Y_n\circ T^k}{n} \leq c(n) \bigg)
  \bigg(=
  \mu_{\hat{T}^k H}\bigg(\frac{Z^Y_n}{n} \leq c(n) \bigg)
  \bigg)
  \sim
 \frac{\sin(\pi\alpha)}{\pi\alpha}\frac{c(n)^\alpha \ell(n)}{\ell(c(n)n)}
 \quad(n\to\infty).  	
\end{align}
\end{Cor}

\begin{proof}[Proof of Corollary $\ref{Cor1:Z}$ by using Theorem $\ref{Thm:Z}$]
Note that $Z^Y_{n}\circ T^k= \max\{0,\;Z^Y_{n+k}-k\}$, and hence $\{Z^Y_n\circ T^k \leq nc(n)\}=\{Z^Y_{n+k}\leq nc(n)+k\}$.
Therefore Theorem \ref{Thm:Z} yields
\begin{align}
	\mu_{H}\bigg(\frac{Z^Y_n\circ T^k}{n} \leq c(n) \bigg)
	&=
	\mu_{H}\bigg(\frac{Z^Y_{n+k}}{n}\leq c(n)+\frac{k}{n}\bigg)
	\\
	&=
	\mu_H\bigg(\frac{Z^Y_{n+k}}{n+k}\leq \frac{nc(n)+k}{n+k}\bigg)
	\\
	&\sim
	\frac{\sin(\pi\alpha)}{\pi\alpha}
	\bigg(\frac{c(n)n+k}{n+k}\bigg)^\alpha
	\frac{\ell(n+k)}{\ell(c(n)n+k)}
	\\
	&\sim
	\frac{\sin(\pi\alpha)}{\pi\alpha}\frac{c(n)^\alpha \ell(n)}{\ell(c(n)n)}
	\quad
	(n\to\infty).
\end{align}
Here we used the uniform convergence theorem for slowly varying functions. 
This completes the proof.
\end{proof}

\begin{Cor}\label{Cor2:Z}
Suppose that the conditions \ref{Thm:Z-cond-1}, \ref{Thm:Z-cond-2}, \ref{Thm:Z-cond-3}, \ref{Thm:Z-cond-4} in Theorem $\ref{Thm:Z}$ are fulfilled. Let $G\in  \{u\in L^1(\mu)\::\:u\geq0\}$. Then the following assertions hold:
\begin{enumerate}[label=$(\arabic*)$]

\item\label{Cor2:Z-1} Assume that  $G$ is  uniformly sweeping for $1_Y$. Then there exists  $C_1\in(0,\infty)$ such that, for any positive-valued sequence $\{c(n)\}_{n\geq0}$ satisfying \eqref{eq:c(n)}, we have
\begin{equation}\label{eq:Cor2:Z-1}
	C_1 \leq 
	\liminf_{n\to\infty}
	\frac{\ell(c(n)n)}{c(n)^\alpha \ell(n)}
	\mu_G\bigg(\frac{Z^Y_n}{n} \leq c(n) \bigg).
\end{equation}

\item\label{Cor2:Z-2}  Assume that $1_Y$ is uniformly sweeping for $G$.  Then there exists $C_2\in(0,\infty)$ such that, for any positive-valued sequence $\{c(n)\}_{n\geq0}$ satisfying \eqref{eq:c(n)}, we have
\begin{equation}\label{eq:Cor2:Z-2}
	\limsup_{n\to\infty}
	\frac{\ell(c(n)n)}{c(n)^\alpha \ell(n)}
	\mu_G\bigg(\frac{Z^Y_n}{n} \leq c(n) \bigg)
  \leq C_2. 
\end{equation}

\end{enumerate}
\end{Cor}

\begin{proof}[Proof of Corollary $\ref{Cor2:Z}$ by using Theorem $\ref{Thm:Z}$]

\ref{Cor2:Z-1} 
By the assumption, $G$ is also uniformly sweeping for $H$.
Take $K\in\bN$ so large that $\sum_{k=0}^{K-1} \hat{T}^k G\geq K^{-1}H$, a.e.
Let $k\in\{0,1,\dots,K\}$. 
Note that $Z^Y_{n}\circ T^k +K\geq Z^Y_{n}$ and hence
\begin{equation}
    \mu_G
    \bigg(\frac{Z^Y_{n}}{n}\leq c(n)\bigg)
    \geq
     \mu_G
    \bigg(\frac{Z^Y_{n}\circ T^k}{n}\leq c(n)-\frac{K}{n}\bigg).
\end{equation}
Therefore Theorem \ref{Thm:Z} yields that
\begin{equation}
\begin{split}
    \mu_G
    \bigg(\frac{Z^Y_{n}}{n}\leq c(n)\bigg)
    &\geq
     \frac{1}{K}\sum_{k=0}^{K-1}
    \mu_G
    \bigg(\frac{Z^Y_{n}\circ T^k}{n}\leq c(n)-\frac{K}{n}\bigg)
    \\
    &\geq
    \frac{1}{K^2}
    \mu_{H}
    \bigg(\frac{Z^Y_{n}}{n}\leq c(n)-\frac{K}{n}\bigg)
    \\
    &\sim
    \frac{\sin(\pi\alpha)}{K^2\pi\alpha}\frac{c(n)^\alpha \ell(n)}{\ell(c(n)n)}
 \quad (n\to\infty),
\end{split}    
\end{equation}
which implies the desired result. 

\ref{Cor2:Z-2}
By the assumption, $H$ is also uniformly sweeping for $G$.
 Take $K\in\bN$ so large that $G\leq K\sum_{k=0}^{K-1} \hat{T}^k H$, a.e. Then we use Corollary \ref{Cor1:Z} to obtain
\begin{equation}
\begin{split}
	\mu_G
    \bigg(\frac{Z^Y_{n}}{n}\leq c(n)\bigg)
    &\leq
    K
    \sum_{k=0}^{K-1}
    \mu_{H}
    \bigg(\frac{Z^Y_n\circ T^k}{n}\leq c(n)\bigg)
    \\
    &\sim
    K^2
    \frac{\sin(\pi\alpha)}{\pi\alpha}\frac{c(n)^\alpha \ell(n)}{\ell(c(n)n)}
 \quad(n\to\infty).
\end{split}	
\end{equation}
We now complete the proof.
\end{proof}

We will also give the proof of the following theorem in Section \ref{Sec:proof-Z}.

\begin{Thm}\label{Thm:Z-weak}
Suppose that the conditions \ref{Thm:Z-cond-1}, \ref{Thm:Z-cond-2}, \ref{Thm:Z-cond-3} of Theorem $\ref{Thm:Z}$ are fulfilled. 
Let $G\in  \{u\in L^1(\mu)\::\:u\geq0\}$.  Then the assertion \ref{Cor2:Z-2} of Corollary $\ref{Cor2:Z}$ holds.
\end{Thm}

In other words, the assertion \ref{Cor2:Z-2} of Corollary \ref{Cor2:Z} remains valid without assuming the existence of the asymptotic entrance density $H$. 
The reader may expect the assertion \ref{Cor2:Z-1} also remains valid under a similar setting, but we do not know whether it is true. The reason is that $\mu_G(Z_n^Y/n\leq c(n))$ can be bounded above but not below by double Laplace transform of $Z_n^Y$ as we shall see in the proof of Theorem \ref{Thm:Z-weak}.

\begin{Rem}
\label{Rem:limsup=infty}
Suppose that the condition \ref{Thm:Z-cond-1} of Theorem $\ref{Thm:Z}$ is satisfied. Let $\alpha\in(0,1)$ and
let $\{c(n)\}_{n\geq0}$ be a non-increasing, $(0,1]$-valued sequence satisfying $c(0)=1$ and $c(n)\to0$ $(n\to\infty)$.
Then there exists a $\mu$-probability density function $G$ such that 
\begin{align}\label{limsup-Z}
    \limsup_{n\to\infty}\frac{1}{c(n)^\alpha}\mu_{G}\bigg(\frac{Z_n^Y}{n}\leq c(n)\bigg)=\infty,	
\end{align}
as we shall see below.
Indeed, let
\begin{align}
   N_0=0\quad \text{and}\quad 
   N_{k}=\min\{n> N_{k-1}\;:\; \mu(Y\cap\{\varphi=n\})>0\}
   \quad(k\in \bN).	
\end{align}
Then $\{N_k\}_{k\geq0}\subset\bN_0$ is strictly increasing.
We define the $\mu$-probability density function $G:X\to[0,\infty)$ as 
\begin{align}
   G=
   \begin{cases}
   	(c(N_{k-1})^{\alpha/2}-c(N_{k})^{\alpha/2})/\mu(Y\cap\{\varphi=N_k\}),
   	&\text{on $Y\cap \{\varphi=N_k\}$ $(k\in\bN)$,}
   	\\
   	0,
   	&\text{otherwise.}
   \end{cases}
\end{align}
Then $ \mu_G(\varphi> N_k)=c(N_k)^{\alpha/2}$ and hence
\begin{align}
    \frac{1}{c(N_k)^\alpha}
    \mu_G\bigg(\frac{Z^Y_{N_k}}{N_k}
    \leq c(N_k)
    \bigg)	
    \geq
    \frac{1}{c(N_k)^\alpha}
    \mu_G(\varphi> N_k)
    =
    c(N_k)^{-\alpha/2}
    \to\infty\quad(k\to\infty),
\end{align}
which implies \eqref{limsup-Z}.
    	
\end{Rem}

%
%
 
\subsection{Large deviation estimates related to the Darling--Kac theorem}

For $A\in \cA$, let $S_n^{A}(x)$ denote the occupation time on $A$ of the orbit $\{T^k x\}_{k\geq0}$ from time $1$ to time $n$, i.e.,
\begin{align}
S_n^{A}(x)=\sum_{k=1}^n 1_{A}(T^k x)\quad (n\in\bN_0,\;x\in X).
\end{align}
In the following we consider occupation times on a set of finite measure.

\begin{Thm}\label{Thm:SY}
Suppose the following conditions \ref{Thm:SY-cond-1}, \ref{Thm:SY-cond-2}, \ref{Thm:SY-cond-3} are satisfied:
\begin{enumerate}[label=\textup{(B\arabic*)}]
	\item\label{Thm:SY-cond-1} $(X, \cA, \mu)$ is a $\sigma$-finite measure space with $\mu(X)=\infty$, $Y\in \cA$, $\mu(Y)\in(0,\infty)$, and $T$ is a CEMPT on $(X,\cA, \mu)$.

    \item\label{Thm:SY-cond-2} $w_n^Y\sim n^{1-\alpha}\ell(n)$ $(n\to\infty)$ for some $\alpha\in(0,1)$ and some positive, measurable function $\ell:(0,\infty)\to(0,\infty)$ slowly varying at $\infty$. 
	\item\label{Thm:SY-cond-3} $\displaystyle 
	\bigg\{\frac{1}{w^{Y}_n}\sum_{k=0}^{n-1}\hat{T}^k 1_{Y_k}\bigg\}_{n\geq1}$ is $L^\infty(\mu)$-bounded.
\end{enumerate}
For $t>0$, set
\begin{align}
    a(t)
    = 
     \frac{t}{\Gamma(1+\alpha)Q^Y(t^{-1})}
     \sim 
    \frac{t^{\alpha}}{\Gamma(1+\alpha)\Gamma(2-\alpha)\ell(t)}
    \quad (t\to\infty).	
\end{align}
Let $\{\tilde{c}(n)\}_{n\geq0}$ be a positive-valued sequence satisfying
\begin{equation}\label{eq:tilde-c(n)}
	\tilde{c}(n)\to0\quad\text{and}\quad
	\tilde{c}(n)a(n)\to\infty
	\quad(n\to\infty).
\end{equation}
Then
\begin{align}\label{eq:Thm:SY}
  \mu_{1_Y}\bigg(\frac{S^{Y}_n}{a(n)} \leq \tilde{c}(n) \bigg)
  \sim
  \frac{\sin(\pi\alpha)}{\pi\alpha}\tilde{c}(n)
 \quad(n\to\infty).  	
\end{align}
\end{Thm}

The proof of Theorem \ref{Thm:SY} will be given in Section \ref{Sec:proof-SY}.

\begin{Rem}[Comparison to the Darling--Kac theorem]
Let us recall the Darling--Kac theorem.
Set
\begin{align} 
 F(t)=\frac{1}{\pi\alpha}\int_0^t
  \sum_{k=1}^\infty 
  \frac{(-1)^{k-1}}{k!} \sin(\pi\alpha k)\Gamma(1+\alpha k)s^{k-1}\,ds
  \quad
  (t\geq0), 
\end{align}
which is the distribution function of the Mittag-Leffler distribution of order $\alpha$ with Laplace transform
\begin{align}
	\int_0^\infty 
	e^{-\lambda t}\,dF(t)
  =
  \sum_{k=0}^\infty\frac{(-\lambda)^k}{\Gamma(1+\alpha k)}\quad (\lambda\in \bR).
\end{align}
See \cite{Pol48, DK} for the details. 
As a special case, the Mittag-Leffler distribution of order $1/2$ is the half-normal distribution with mean $2/\sqrt{\pi}$.
 Suppose the conditions \ref{Thm:Z-cond-1}, \ref{Thm:Z-cond-2}, \ref{Thm:Z-cond-3}, \ref{Thm:Z-cond-4} of Theorem \ref{Thm:Z} are satisfied. Then for any $\mu$-absolutely continuous probability measure $\nu$ on $X$ and for any $t\geq0$, we have
\begin{align}
 \label{eq:DK}
  \lim_{n\to\infty}\nu\bigg(\frac{S^{Y}_n}{a(n)} \leq t \bigg)
 =
 F\bigg(\frac{t}{\Gamma(1+\alpha)\mu(Y)}\bigg),
\end{align}
which follows from \cite[Theorem 3.1]{ThZw}. See also \cite[Theorem 2.1]{Zwe07cpt} and \cite[Theorem 2.1]{KocZwe}.
Note that
\begin{equation}
	 F\bigg(\frac{\tilde{c}(n)}{\Gamma(1+\alpha)\mu(Y)}\bigg)
	 \sim
	 \frac{\sin(\pi\alpha)}{\pi\alpha}\frac{\tilde{c}(n)}{\mu(Y)}
	 \quad
	 (n\to\infty).
\end{equation}
Nevertheless \eqref{eq:Thm:SY} does not follow from \eqref{eq:DK} directly.
\end{Rem}

\begin{Cor}\label{Cor1:SY}
Let $k\in \bN_{0}$. Under the setting of Theorem $\ref{Thm:SY}$,
\begin{equation}
	\mu_{1_Y}\bigg(\frac{S^{Y}_n\circ T^k}{a(n)} \leq \tilde{c}(n) \bigg)
	\bigg(=
	\mu_{\hat{T}^k 1_Y}\bigg(\frac{S^{Y}_n}{a(n)} \leq \tilde{c}(n) \bigg)
	\bigg)
	\sim
	\frac{\sin(\pi\alpha)}{\pi\alpha}\tilde{c}(n)
	\quad
	(n\to\infty).
\end{equation}	
\end{Cor}

\begin{proof}[Proof of Corollary $\ref{Cor1:SY}$ by using Theorem $\ref{Thm:SY}$]
Since $|S^Y_n-S^Y_n\circ T^k|\leq k$, we see that
\begin{equation}
    \mu_{1_Y}\bigg(\frac{S^{Y}_n}{a(n)} \leq \tilde{c}(n)-\frac{k}{a(n)}\bigg)
    \leq
    \mu_{1_Y}\bigg(\frac{S^{Y}_n\circ T^k}{a(n)} \leq \tilde{c}(n) \bigg)
    \leq
    \mu_{1_Y}\bigg(\frac{S^{Y}_n}{a(n)} \leq \tilde{c}(n) +\frac{k}{a(n)} \bigg).	
\end{equation}
By Theorem \ref{Thm:SY}, 
\begin{equation}
	\mu_{1_Y}\bigg(\frac{S^{Y}_n}{a(n)} \leq \tilde{c}(n) \pm\frac{k}{a(n)} \bigg)\sim
	\frac{\sin(\pi\alpha)}{\pi\alpha}\tilde{c}(n)
 \quad(n\to\infty).
\end{equation}
Therefore we obtain the desired result.
\end{proof}

\begin{Cor}\label{Cor2:SY}
Suppose that the conditions \ref{Thm:SY-cond-1}, \ref{Thm:SY-cond-2}, \ref{Thm:SY-cond-3} of Theorem $\ref{Thm:SY}$ are fulfilled. Let $G\in \{u\in L^1(\mu)\::\:u\geq0\}$. Then the following assertions hold:
\begin{enumerate}[label=$(\arabic*)$]

\item\label{Cor2:SY-1} Assume that $G$ is uniformly sweeping for $1_Y$. Then there exists $C_1\in(0,\infty)$ such that, for any positive-valued sequence $\{\tilde{c}(n)\}_{n\geq0}$ satisfying \eqref{eq:tilde-c(n)}, we have
\begin{equation}
	C_1 \leq \liminf_{n\to\infty}\frac{1}{\tilde{c}(n)}\mu_G\bigg(\frac{S_n^Y}{a(n)} \leq \tilde{c}
	(n) \bigg).
	\label{eq:Cor2:SY-1}
\end{equation}

\item\label{Cor2:SY-2} Assume that $1_Y$ is uniformly sweeping for $G$.  Then there exists $C_2\in(0,\infty)$ such that, for any positive-valued sequence $\{\tilde{c}(n)\}_{n\geq0}$ satisfying \eqref{eq:tilde-c(n)}, we have
\begin{equation}
	\limsup_{n\to\infty}\frac{1}{\tilde{c}(n)}\mu_G\bigg(\frac{S_n^{Y}}{a(n)} \leq \tilde{c}(n)\bigg)
  \leq C_2.
  \label{eq:Cor2:SY-2} 
\end{equation}

\end{enumerate}	
\end{Cor}

\begin{proof}[Proof of Corollary $\ref{Cor2:SY}$ by using Theorem $\ref{Thm:SY}$]

\ref{Cor2:SY-1} Take $K\in\bN$ so large that $\sum_{k=0}^{K-1} \hat{T}^k G\geq K^{-1}1_Y$, a.e.
Let $k\in\{0,1,2,\dots,K\}$. Then
\begin{equation}
 	\mu_G\bigg(\frac{S^{Y}_n}{a(n)} \leq \tilde{c}(n) \bigg)
 	\geq
 	\mu_G\bigg(\frac{S^{Y}_n\circ T^k}{a(n)} \leq \tilde{c}(n)-\frac{K}{a(n)}\bigg).
\end{equation}
Hence Theorem \ref{Thm:SY} implies
\begin{equation}
\begin{split}
	\mu_G\bigg(\frac{S^{Y}_n}{a(n)} \leq \tilde{c}(n) \bigg)
	&\geq
	K^{-1}\sum_{k=0}^{K-1}
	\mu_G\bigg(\frac{S^{Y}_n\circ T^k}{a(n)} \leq \tilde{c}(n)-\frac{K}{a(n)}\bigg)
	\\
	&\geq
	K^{-2}
	\mu_{1_Y}\bigg(\frac{S^{Y}_n}{a(n)} \leq \tilde{c}(n)-\frac{K}{a(n)}\bigg)
	\\
	&\sim
	\frac{\sin(\pi\alpha)}{K^{2}\pi\alpha}\tilde{c}(n)
	\quad(n\to\infty).
\end{split}	
\end{equation}
which implies the desired result.

\ref{Cor2:SY-2}  Take $K\in\bN$ large enough that $G\leq K\sum_{k=0}^{K-1} \hat{T}^k 1_Y$, a.e. Then we use Corollary \ref{Cor1:SY} to obtain
\begin{equation}
\begin{split}
\mu_G\bigg(\frac{S^{Y}_n}{a(n)} \leq \tilde{c}(n) \bigg)
&\leq
K\sum_{k=0}^{K-1}\mu_{1_Y}\bigg(\frac{S^{Y}_n\circ T^k}{a(n)} \leq \tilde{c}(n)\bigg)
\\
&\sim
K^2\frac{\sin(\pi\alpha)}{\pi\alpha}\tilde{c}(n)
\quad(n\to\infty).
\end{split}	
\end{equation}
We now complete the proof.
\end{proof}

\subsection{Large deviation estimates related to the Lamperti  generalized arcsine law}

In the following we consider occupation times on sets of infinite measure under certain additional assumptions. Fix  disjoint sets $Y, A_1, A_2,\dots,A_d \in \cA$ with $d\in\bN$, $d\geq2$, $X=Y\cup \bigcup_{i=1}^d A_i$, $0<\mu(Y)<\infty$ and $\mu(A_i)=\infty$ $(i=1,2,\dots,d)$. 
We assume \emph{$Y$ dynamically separates $A_1, A_2,\dots,A_d$   (under the action of $T$)}, that is, $A_i\cap T^{-1}A_j =\emptyset$ whenever $i\neq j$.  
Then the condition [$x\in A_i$ and $T^n x\in A_j$ $(i\neq j)$] implies $n\geq2$ and the existence of $k=k(x)\in\{1,\dots,n-1\}$ for which $T^k x\in Y$. 
As shown in  \cite[(6.6)]{ThZw},
\begin{align}
	 \label{eq:1YnAi}
	1_{Y_n\cap A_i}
	&=
	\sum_{k>n}\hat{T}^{k-n}1_{Y\cap (T^{-1}A_i) \cap\{\varphi=k\}}
	\quad
	\text{a.e.}\;\;(n\in\bN, \; i=1,2,\dots,d).
\end{align}
and $\mu(Y_n\cap A_i)=\mu(Y\cap (T^{-1}A_i) \cap \{\varphi>n\})$ ($n\in\bN$, $i=1,2,\dots,d$).
Let $\{w^{Y, A_i}_n\}_{n\geq0}$ denote the \emph{wandering rate of $Y$ starting from $A_i$}, which is given by
\begin{align}
    w^{Y, A_i}_n
    &=
    \mu\bigg(
    \bigcup_{k=0}^{n-1} (T^{-k}Y)\cap A_i
    \bigg)
    =
    \sum_{k=0}^{n-1}
    \mu(Y_k\cap A_i)
    \\
    &=
    \sum_{k=1}^{n-1}
    \mu(Y\cap (T^{-1}A_i) \cap \{\varphi>k\})
    \quad
    (n\in\bN_0,\;i=1,2,\dots,d).
    \label{eq:wYAi}
\end{align}
We write $Q^{Y, A_i}(s)$ for Laplace transform of $\{w^{Y,A_i}_{n+1}-w^{Y, A_i}_n\}_{n\geq0}$:
\begin{align}
	Q^{Y,A_i}(s)
	&=\sum_{n\geq0} e^{-ns}\big(w_{n+1}^{Y,A_i}-w_{n}^{Y,A_i}\big)
	\\
	&=\sum_{n\geq1}e^{-ns}\mu(Y\cap (T^{-1}A_i)\cap\{\varphi>n\})\quad(s>0,\;i=1,2,\dots,d).	
\end{align}
Then $w^{Y}_n=\mu(Y)+\sum_{i=1}^d w^{Y, A_i}_n$ and $Q^Y(s)=\mu(Y)+\sum_{i=1}^d Q^{Y,A_i}(s)$.
Let $\alpha, \beta_1,\beta_2,\dots,\beta_d\in (0,1)$ with $\sum_{i=1}^d \beta_i=1$ and let $\ell:(0,\infty)\to(0,\infty)$ be a positive, measurable function slowly varying at $\infty$. By Karamata's Tauberian theorem, the condition 
\begin{equation}
	w_n^{Y,A_i}
	\sim
	\beta_i n^{1-\alpha}\ell(n)
	\quad
	(n\to\infty,\;i=1,2,\dots,d)
\end{equation}
is equivalent to
\begin{align}\label{eq:RV:Qi}
	 Q^{Y,A_i}(s)\sim \Gamma(2-\alpha) \beta_i s^{-1+\alpha} \ell(s^{-1})
	 \quad
	 (s\to 0+,\;i=1,2,\dots,d).
\end{align}

\begin{Thm}\label{Thm:SA}
Suppose the following conditions \ref{Thm:SA-cond-1}, \ref{Thm:SA-cond-2}, \ref{Thm:SA-cond-3}, \ref{Thm:SA-cond-4}, \ref{Thm:SA-cond-5} are satisfied:
\begin{enumerate}[label=\textup{(C\arabic*)}]
	\item\label{Thm:SA-cond-1} $T$ is a CEMPT on a $\sigma$-finite measure $(X,\cA, \mu)$, and $X=Y\cup \bigcup_{i=1}^d  A_i$ for some $d\in\bN$, $d\geq2$ and some disjoint sets $Y, A_1, A_2,\dots,A_d \in \cA$ satisfying $0<\mu(Y)<\infty$ and $\mu(A_i)=\infty$ $(i=1,2,\dots,d)$. In addition $Y$ dynamically separates $A_1, A_2,\dots,A_d$.

	\item\label{Thm:SA-cond-2} $w^{Y, A_i}_n\sim \beta_i n^{1-\alpha}\ell(n)$ $(n\to\infty,\; i=1,2,\dots,d)$ for some $\alpha, \beta_1, \beta_2,\dots,\beta_d \in(0,1)$ with $\sum_{i=1}^d \beta_i=1$ and for some positive, measurable function $\ell:(0,\infty)\to(0,\infty)$ slowly varying at $\infty$.
\item \label{Thm:SA-cond-3}
    $\displaystyle 
	\bigg\{
	\frac{1}{w_n^{Y,A_d}}
	\sum_{k=0}^{n-1}\hat{T}^k 1_{Y_k\cap A_d}
	\bigg\}_{n\geq2}$ is $L^\infty(\mu)$-bounded. 
\item \label{Thm:SA-cond-4} There exists $N\geq2$ such that $\displaystyle \bigg\{\frac{1}{w_n^{Y,A_i}}\sum_{k=0}^{n-1}\hat{T}^k 1_{Y_k\cap A_i}\bigg\}_{n\geq N,\;i=1,\dots,d-1}$ is uniformly sweeping for $1_Y$.

\item\label{Thm:SA-cond-5}  There exist  $H^{(1)},\dots,H^{(d-1)}\in L^\infty(\mu)$  such that
\begin{equation}
	\lim_{n\to\infty}\bigg(\frac{1}{w_n^{Y,A_i}}
	 \sum_{k=0}^{n-1} \hat{T}^k 1_{Y_k\cap A_i}\bigg)
   = H^{(i)}
   \quad \text{in $L^\infty(\mu)$}
   \;\;(i=1,\dots,d-1).
\end{equation}

\end{enumerate}
Let $\lambda_1,\dots,\lambda_{d-1}\in(0,\infty)$ and let $\{c(n)\}_{n\geq0}$ be a positive-valued sequence satisfying  \eqref{eq:c(n)}.
Set $\tilde{H}=\sum_{i=1}^{d-1}\beta_i \lambda_i^\alpha H^{(i)}$.
Then 
\begin{equation}\label{eq:Thm:SA}
   \mu_{\tilde{H}}\bigg(\frac{1}{n}\sum_{i=1}^{d-1}\lambda_i S_n^{A_i}\leq c(n)\bigg)
   \sim
   \beta_d \frac{\sin(\pi \alpha)}{\pi\alpha}
   \frac{c(n)^\alpha \ell(n)}{\ell(c(n)n)}
   \quad
   (n\to\infty).	
\end{equation}
\end{Thm}

The proof of Theorem \ref{Thm:SA} will be given in Section \ref{Sec:proof-SA}.

\begin{Rem}
Under the setting of Theorem \ref{Thm:SA} with $d=2$ and $\lambda_{1}=1$, let us assume \eqref{eq:c(n)-2}. Then we obtain
\begin{align}\label{eq:Thm:SA-modif}
   	 \mu_{H^{(1)}}\bigg(\frac{S_n^{A_1}}{n}\leq c(n)\bigg)
   \sim
   \frac{\beta_2}{\beta_1} \frac{\sin(\pi \alpha)}{\pi\alpha}
   c(n)^\alpha
   \quad
   (n\to\infty).
\end{align}	
\end{Rem}

\begin{Rem}
If $\lambda_i=0$ for some $i$, then \eqref{eq:Thm:SA} does not remain valid. For example, let $d\geq3$,  $\lambda_1, \dots, \lambda_{d-2}\in (0,\infty)$ and $\lambda_{d-1}=0$. Then
\begin{align}
	\mu_{\tilde{H}}\bigg(\frac{1}{n}\sum_{i=1}^{d-2}\lambda_i S_n^{A_i}\leq c(n)\bigg)
   \sim
   (\beta_{d-1}+\beta_d) \frac{\sin(\pi \alpha)}{\pi\alpha}
    \frac{c(n)^\alpha \ell(n)}{\ell(c(n)n)}
   \quad
   (n\to\infty),
\end{align}
which follows from Theorem \ref{Thm:SA}.	
\end{Rem}


\begin{Rem}[Comparison to the Lamperti generalized arcsine law]
Let us recall the Lamperti generalized arcsine law for occupation times.
Suppose the  conditions \ref{Thm:SA-cond-1}, \ref{Thm:SA-cond-2}, \ref{Thm:SA-cond-5} of Thereom \ref{Thm:SA} and \ref{Thm:Z-cond-3} of Theorem \ref{Thm:Z} are fulfilled with $d=2$. 
Set $b=\beta_2/\beta_1$.
Then, for any $\mu$-absolutely continuous probability measure $\nu$ on $(X,\cA)$ and for any $0\leq t\leq 1$, we have
\begin{align}
	   \lim_{n\to\infty}\nu\bigg(\frac{S_n^{A_1}}{n} \leq t\bigg)
  &=
   \frac{b\sin(\pi\alpha)}{\pi}
   \int_0^t 
    \frac{s^{\alpha-1} (1-s)^{\alpha-1}\,ds}
    {b^2s^{2\alpha} +2bs^\alpha (1-s)^\alpha \cos(\pi\alpha)  
    +(1-s)^{2\alpha}}
\\
&=
\frac{1}{\pi\alpha} \mathop{\rm arccot}\bigg(\frac{(1-t)^\alpha}{b  \sin(\pi\alpha) t^\alpha}+\cot(\pi\alpha)\bigg),
\label{eq:Arcsine-S}
\end{align}
as shown in \cite[Theorem 2.2]{Zwe07cpt}.
See also \cite[Theorem 3.2]{ThZw} and \cite[Theorem 2.7]{SerYan19}.
The limit is the distribution function of the Lamperti generalized arcsine distribution of parameter $(\alpha,\beta_1)$. In the case $\alpha=\beta_1=\beta_2=1/2$, this distribution is the usual arcsine distribution.
Note that the right-hand side of \eqref{eq:Thm:SA-modif} is asymptotically the same as the right-hand side of \eqref{eq:Arcsine-S} with $t=c(n)$.
Nevertheless \eqref{eq:Thm:SA-modif} does not follow from \eqref{eq:Arcsine-S} directly.
\end{Rem}

The proofs of the following two corollaries are almost the same as those of Corollaries \ref{Cor1:Z}, \ref{Cor2:Z}, \ref{Cor1:SY} and \ref{Cor2:SY}, so we omit them. 

\begin{Cor}\label{Cor1:SA}
Let $k\in\bN_0$. Under the setting of Theorem $\ref{Thm:SA}$, 
\begin{align}
 \mu_{\tilde{H}}\bigg(\frac{1}{n}\sum_{i=1}^{d-1}\lambda_i S_n^{A_i}\circ T^k \leq c(n)\bigg)
  \sim
 \beta_d \frac{\sin(\pi\alpha)}{\pi\alpha}\frac{c(n)^\alpha\ell(n)}{\ell(c(n)n)}
 \quad(n\to\infty).  	
\end{align}
\end{Cor}

\begin{Cor}\label{Cor2:SA}
Suppose that the conditions \ref{Thm:SA-cond-1}, \ref{Thm:SA-cond-2}, \ref{Thm:SA-cond-3}, \ref{Thm:SA-cond-4}, \ref{Thm:SA-cond-5} of Theorem $\ref{Thm:SA}$ are satisfied. Let $G\in\{u\in L^1(\mu)\::\:u\geq0\}$ and $\lambda_1,\dots,\lambda_{d-1}\in (0,\infty)$. Then the following assertions hold:
\begin{enumerate}[label=$(\arabic*)$]

\item\label{Cor2:SA-1} Assume that  $G$ is uniformly sweeping for $1_Y$. Then there exists  $C_1\in(0,\infty)$ such that, for any positive-valued sequence $\{c(n)\}_{n\geq0}$ satisfying \eqref{eq:c(n)}, we have
\begin{equation}\label{eq:Cor2:SA-1}
	C_1 \leq 
	\liminf_{n\to\infty}
	\frac{\ell(c(n)n)}{c(n)^\alpha\ell(n)}
	\mu_G\bigg(\frac{1}{n}\sum_{i=1}^{d-1}\lambda_i S_n^{A_i} \leq c(n) \bigg).
\end{equation}

\item\label{Cor2:SA-2} Assume that $1_Y$ is uniformly sweeping for $G$.  Then there exists $C_2\in(0,\infty)$ such that, for any positive-valued sequence $\{c(n)\}_{n\geq0}$ satisfying \eqref{eq:c(n)}, we have
\begin{equation}\label{eq:Cor2:SA-2}
	\limsup_{n\to\infty}
		\frac{\ell(c(n)n)}{c(n)^\alpha\ell(n)}
	\mu_G\bigg(\frac{1}{n}\sum_{i=1}^{d-1}\lambda_i S_n^{A_i}\leq c(n) \bigg)
  \leq C_2. 
\end{equation}
\end{enumerate}
\end{Cor}

The assertion \ref{Cor2:SA-2} remains valid even if we drop the condition \ref{Thm:SA-cond-5} of Theorem \ref{Thm:SA}:

\begin{Thm}\label{Thm:SA-weak}
Suppose that the conditions \ref{Thm:SA-cond-1}, \ref{Thm:SA-cond-2}, \ref{Thm:SA-cond-3}, \ref{Thm:SA-cond-4} of Theorem $\ref{Thm:SA}$ are satisfied. Let $G\in  \{u\in L^1(\mu)\::\:u\geq0\}$ and $\lambda_1,\dots,\lambda_{d-1}\in (0,\infty)$. Then the assertion \ref{Cor2:SA-2} of Corollary $\ref{Cor2:SA}$ holds.
\end{Thm}

We will give the proof of Theorem \ref{Thm:SA-weak} in Section \ref{Sec:proof-SA}.

\section{Analytical tools}\label{Sec:Ana}

In this section we prove lemmas needed in the sequel.


\begin{Lem}\label{Lem:Laplace}
Let $f_n:(0,\infty)\to [0,\infty)$ $(n\in \bN\cup \{\infty\})$ be non-increasing functions. Assume there exists a non-empty open interval $I\subset(0,\infty)$ such that for any $q\in I$,
\begin{align}
 \lim_{n\to\infty}\int_0^\infty e^{-qu}f_n(u)\,du 
 =
  \int_0^\infty e^{-qu}f_\infty(u)\,du<\infty.
 \end{align}
Then $\lim_{n\to\infty}f_n(u)= f_\infty(u)$ for all continuity points $u\in (0,\infty)$ of $f_\infty$.
\end{Lem}

See \cite[Lemma 3.2]{SerYan19} for the proof of Lemma \ref{Lem:Laplace}.


\begin{Lem}\label{Lem:discrete}
Fix a constant $C>0$. Let $S_n:X\to[0,Cn]$ $(n\in\bN_0)$ be measurable functions defined on a measure space $(X,\cA)$, and let $\lambda:(0,\infty)\to[0,\infty)$ be a non-negative function with $\lambda(t)\to0$ $(t\to\infty)$. 
Suppose $\nu_t$ $(t>0)$ are non-zero finite measures on $X$.
Then for any  $q>0$,
\begin{align}
	&\sum_{n=0}^\infty e^{-nq t^{-1}} \int_X \exp(-\lambda(t) S_{n})\, d\nu_t
	\\[5pt]
	&\sim
	t\int_0^\infty du\, e^{-qu}\int_X d\nu_t\, \exp(-\lambda(t) S_{[ut]})
	\quad(t\to\infty).
	\label{Lem:DL-1-1}	
\end{align}
\end{Lem}
\begin{proof}
Note that
\begin{align}
	\text{(the right-hand side of \eqref{Lem:DL-1-1})}
	&=
	\sum_{n=0}^\infty
	t
	\int_{nt^{-1}}^{(n+1)t^{-1}}
	du\,e^{-qu}\int_X d\nu_t\, \exp(-\lambda(t)S_{n}),
	\label{Lem:DL-1-2}
\end{align}
and hence
\begin{equation}
\begin{split}
	0
	&\leq 
	\text{(the left-hand side of \eqref{Lem:DL-1-1})}
	-
	\text{(the right-hand side of \eqref{Lem:DL-1-1})}
	\\
	&\leq
	\sum_{n=0}^\infty (e^{-nq t^{-1}}-e^{-(n+1)q t^{-1}})\nu_t(X)
	=
	\nu_t(X).
\end{split}	
\end{equation}
In addition, since $0\leq S_{[ut]}\leq Cut$, we have
\begin{align}
	\text{(the right-hand side of \eqref{Lem:DL-1-1})}
	&\geq
	\nu_t(X)t\int_0^\infty \exp(-(q+C\lambda(t)t) u)\,du
	\\
	&=
	\frac{\nu_t(X)}{qt^{-1}+C\lambda(t)}.
\end{align}
Therefore we obtain
\begin{align}
	1
	\leq
	\frac{\text{(the left-hand side of \eqref{Lem:DL-1-1})}}
	     {\text{(the right-hand side of \eqref{Lem:DL-1-1})}}
	\leq
	1+qt^{-1}+C\lambda(t)
	\to
	1
	\quad
	(t\to\infty).
\end{align}
We now complete the proof.
\end{proof}

The following three lemmas are slight extensions of \cite[Lemma 4.2]{ThZw}. 

\begin{Lem}\label{Lem:IT-weak}
Fix a constant $t_0>0$. Suppose that the following conditions \ref{Lem:IT-1},  \ref{Lem:IT-2},  \ref{Lem:IT-3} are fulfilled:
\begin{enumerate}[label=\textup{(\roman*)}]
	\item\label{Lem:IT-1} $T$ is a CEMPT on a $\sigma$-finite measure space $(X, \cA, \mu)$. 
	
\item\label{Lem:IT-2} $\{H_t\}_{t>0}\cup\{G\}\subset \{u\in L^1(\mu):u\geq0\}$. In addition,  $\{H_t\}_{t\geq t_0}$ is uniformly sweeping in $K$ steps for $G$.
\item\label{Lem:IT-3} $R_{n,t}:X\to(0,1]$ $(n\in \bN_0,\;t>0)$  are measurable functions with
\begin{equation}
	\sup\bigg\{\bigg\|\frac{R_{n,t}\circ T^k}{R_{n+k,t}}\bigg\|_{L^\infty(\mu)}\::\:n,k\in \bN_0,\;0\leq k\leq K,\; t\geq t_0\bigg\}<\infty.
\end{equation}
\end{enumerate}
Then, for any $q>0$, we have
\begin{align}
   \sup_{t\geq t_0}
   \frac
   {
   \sum_{n\geq0} e^{-nqt^{-1}}
   \int_X 
    R_{n,t}G\,d\mu
	}
	{
	\sum_{n\geq0} e^{-nqt^{-1}}
	\int_X
	R_{n,t}
	H_t \,d\mu}
	<\infty.
\end{align}
\end{Lem}

\begin{proof}
Take $C>0$ large enough so that, for any $n,k\in\bN_0$ with $0\leq k\leq K$ and for any $t\geq t_0$, we have $G\leq C\sum_{k=0}^K \hat{T}^k H_t$ and $R_{n,t}\circ T^k\leq CR_{n+k,t}$ a.e.
Then
\begin{align}
    \int_X R_{n,t}G\,d\mu
    &\leq
    C\int_X R_{n,t}\bigg(\sum_{k=0}^K \hat{T}^k H_t\bigg)\,d\mu
    =
    C\sum_{k=0}^K \int_X (R_{n,t}\circ T^k)H_t \,d\mu
    \\
    &\leq
    C^2	\sum_{k=0}^K \int_X R_{n+k,t}H_t\,d\mu
    \quad(n\in\bN_0,\;t\geq t_0)
    ,
\end{align}
which implies
\begin{align}
   \sum_{n\geq0} e^{-nqt^{-1}}
   \int_X 
    R_{n,t}G\,d\mu
    \leq
    \bigg(
    C^2\sum_{k=0}^{K}e^{kqt_0^{-1}}
    \bigg)
    \bigg(
    \sum_{n\geq0} e^{-nqt^{-1}}
   \int_X 
    R_{n,t}H_t\,d\mu\bigg)
    \quad(t\geq t_0).
\end{align}
We now complete the proof.
\end{proof}

\begin{Lem}[Integrating transforms]\label{Lem:IT}
Under the assumptions of Lemma \textup{\ref{Lem:IT-weak}} with $G=1_Y$, $Y\in \cA$, $\mu(Y)\in(0,\infty)$, we further suppose 
\begin{enumerate}[label=\textup{(\roman*)}]
\setcounter{enumi}{3}
\item\label{Lem:IT-4} $\{H_t\}_{t>0}\cup\{H\}\subset \{u\in L^\infty(\mu):u\geq0\; \text{and $u$ is supported on $Y$}\}$ and $H_t\to H$ in $L^\infty(\mu)$ $(t\to\infty)$.
\end{enumerate}  
Then, for any $q>0$,
\begin{align}\label{eq:Lem:IT}
	&\sum_{n\geq0}e^{-nqt^{-1}}
	\int_Y
	R_{n,t}H_t
	\,d\mu
    \sim
	\sum_{n\geq0}
	e^{-nqt^{-1}}
	\int_Y R_{n,t}H\,d\mu
	\quad(t\to\infty).
\end{align}
\end{Lem}
\begin{proof}
By Lemma \ref{Lem:IT-weak},
\begin{align}
   \bigg|
   \frac
   {\sum_{n\geq0}e^{-nqt^{-1}}
	\int_Y
	R_{n,t}H
	\,d\mu}
   {\sum_{n\geq0}e^{-nqt^{-1}}
	\int_Y
	R_{n,t}H_t
	\,d\mu}
   -1\bigg|
   &\leq
   \frac
   {
   \sum_{n\geq0} e^{-nqt^{-1}}
   \int_Y 
    R_{n,t}\,d\mu
	}
	{
	\sum_{n\geq0} e^{-nqt^{-1}}
	\int_Y
	R_{n,t}H_t
	\,d\mu}
	\|H-H_t\|_{L^\infty(\mu)}
	\\
	&\to0
	\quad(t\to\infty),	
\end{align}
as desired.
\end{proof}

\begin{Lem}\label{Lem:Ht}
Let $T$ be a CEMPT on a $\sigma$-finite space $(X,\cA,\mu)$ with $\mu(X)=\infty$, $Y\in\cA$ and $\mu(Y)\in(0,\infty)$. Let $\{v_n\}_{n\geq0}\subset\{u\in L^\infty(\mu)\::\:u\geq0\;\text{and $u$ is supported on $Y$}\}$ with
\begin{align}
    0<\sum_{n\geq0}e^{-ns}\int_Y v_n\,d\mu
    <\infty
    \quad(s>0).	
\end{align}
Let $\lambda:(0,\infty)\to(0,\infty)$ be a positive function with $\lambda(t)\to0$ $(t\to\infty)$.
We define $H_t\in L^1(\mu)$ as
\begin{align}
	H_t
	=
	\frac
	{\sum_{n\geq0}e^{-n\lambda(t)} v_n }
	{\sum_{n\geq0}e^{-n\lambda(t)}\int_Y v_n\,d\mu}
	\,\bigg(
	=
	\frac
	{\sum_{n\geq0}e^{-n\lambda(t)}\sum_{k=0}^n  v_k}
	{\sum_{n\geq0}e^{-n\lambda(t)}\sum_{k=0}^n \int_Y v_k\,d\mu}
	\bigg)
	\quad
	(t>0).
\end{align}
Then the following assertions hold:	
\begin{enumerate}[label=\textup{(\arabic*)}]
\item\label{Lem:Ht-1} 
Assume there exists $N\in \bN_0$ such that $\sum_{k=0}^N \int_Y v_k\,d\mu>0$ and 
\begin{align}
	\bigg\{\frac{\sum_{k=0}^n v_k}{\sum_{k=0}^n \int_Y v_k\,d\mu}\bigg\}_{n\geq N}
\end{align}
is uniformly sweeping in $K$ steps for $1_Y$. Then there exists $t_0>0$ such that $\{H_t\}_{t\geq t_0}$ is uniformly sweeping in $K$ steps for $1_Y$.

\item\label{Lem:Ht-2} 
Assume there exists $H\in L^\infty(\mu)$ such that
\begin{align}
    \frac{\sum_{k=0}^n v_k}{\sum_{k=0}^n \int_Y v_k\,d\mu}
    \to H
    \quad
    \text{in $L^\infty(\mu)$}\;\;(n\to\infty).	
\end{align}
Then $\{H_t\}_{t>0}\subset L^\infty(\mu)$ and $H_t\to H$ in $L^\infty(\mu)$ $(t\to\infty)$.
\end{enumerate}
\end{Lem}

\begin{proof}
\ref{Lem:Ht-1}
Take $C>0$ large enough so that, for any $n\geq N$, 
\begin{align}
C \sum_{k=0}^n \sum_{m=0}^K \hat{T}^m v_k\geq \sum_{k=0}^n \int_Y v_k\,d\mu
\quad
\text{a.e.\, on $Y$.}
\end{align}
Then
\begin{align}
   &C\sum_{m=0}^K \hat{T}^m H_t-1
   \\
   &\geq
   \frac{-\sum_{n=0}^{N-1}\|\sum_{k=0}^n  (\sum_{m=0}^K \hat{T}^m v_k-\int_Y v_k\,d\mu)\|_{L^\infty(\mu)}}{\sum_{n\geq0}e^{-n\lambda(t)}\sum_{k=0}^n \int_Y v_k\,d\mu}
   \quad
   \text{a.e.\, on $Y$.}
   \label{eq:Lem:Ht-1}	
\end{align}
Since $\sum_{n\geq0}e^{-n\lambda(t)}\sum_{k=0}^n \int_Y v_k\,d\mu\to \infty$ $(t\to\infty)$, we can take $t_0>0$ large enough so that, for any $t\geq t_0$, the right-hand side of \eqref{eq:Lem:Ht-1} is greater than $-1/2$. Thus, for any $t\geq t_0$, we have $2C\sum_{m=0}^K \hat{T}^m H_t\geq 1_Y$ a.e., and hence $\{H_t\}_{t\geq t_0}$ is uniformly sweeping in $K$ steps for $1_Y$.

\ref{Lem:Ht-2}
Fix $\varepsilon>0$ arbitrarily. Take $N\in\bN_0$ large enough so that $\sum_{k=0}^N \int_Y v_k\,d\mu>0$ and for any $n\geq N$,
\begin{equation}
    \bigg\|H-\frac{\sum_{k=0}^n v_k}{\sum_{k=0}^n \int_Y v_k\,d\mu}\bigg\|_{L^\infty(\mu)}\leq \varepsilon.
\end{equation}
Then
\begin{align}
   \| H-H_t\|_{L^\infty(\mu)}
   \leq
   \frac{\sum_{n=0}^{N-1}\|\sum_{k=0}^n (H\int_Y v_k\, d\mu-v_k)\|_{L^\infty(\mu)}}
   {\sum_{n\geq0}e^{-n\lambda(t)}\sum_{k=0}^n \int_Y v_k\,d\mu}+\varepsilon\,(<\infty),	
\end{align}
which implies $H_t\in L^\infty(\mu)$ and $\limsup_{t\to\infty}\|H-H_t\|_{L^\infty(\mu)}\leq \varepsilon$. Since $\varepsilon>0$ was arbitrary, we obtain the desired result.
\end{proof}

The following lemma ensures that the condition \ref{Lem:IT-3} in Lemma \ref{Lem:IT-weak} is satisfied for $R_{n,t}=\exp(-\lambda(t) Z^Y_n)$ under the setting of Theorem $\ref{Thm:Z-weak}$. 

\begin{Lem}\label{Lem:bdd-Z}
Suppose that the condition \ref{Thm:Z-cond-1} of Theorem $\ref{Thm:Z}$ holds.
Let $\lambda:(0,\infty)\to[0,\infty)$ be a non-negative function with $\lambda(t)\to0$ $(t\to\infty)$. Set
\begin{equation}
    R_{n,t}=\exp(-\lambda(t)Z^Y_{n})
    \quad(n\in\bN_0,\;t>0).	
\end{equation}
Then there exists a positive constant $t_0>0$ such that for any $n,k\in\bN_0$ and $t\geq t_0$, we have
\begin{equation}
	R_{n,t}\circ T^k \leq e^{k} R_{n+k,t}. 
\end{equation}
\end{Lem}

\begin{proof}
 Take $t_0>0$ so large that $\lambda(t)\leq1$ for any $t\geq t_0$.	
Since $Z^Y_n\circ T^k\geq Z^Y_{n+k}-k$, we have $R_{n,t}\circ T^k \leq \exp(\lambda(t)k) R_{n+k,t}\leq e^k R_{n+k,t}$ for any $t\geq t_0$.
\end{proof}

We can also prove the following lemma in almost the same way, so we omit its proof.

\begin{Lem}\label{Lem:bdd-SA}
Suppose that the condition \ref{Thm:SA-cond-1} of Theorem $\ref{Thm:SA}$ holds.
Let $\lambda_i:(0,\infty)\to[0,\infty)$ $(i=1,\dots,d)$ be non-negative functions with $\lambda_i(t)\to0$ $(t\to\infty,\;i=1,\dots,d)$.
Set
\begin{align}
  R_{n,t}
  =
  \exp\bigg(-\sum_{i=1}^d \lambda_i(t)S_n^{A_i}\bigg)
  \quad(n\in\bN_0,\;t>0).	
\end{align}
Then there exists a positive constant $t_0>0$ such that for any $n,k\in\bN_0$ and $t\geq t_0$, we have
\begin{align}
	R_{n,t}\circ T^k \leq e^{k} R_{n+k,t}.
\end{align}
	
\end{Lem}

\section{Proofs of Theorems \ref{Thm:Z} and \ref{Thm:Z-weak}}
\label{Sec:proof-Z}

In the following lemma, we give a representation of double Laplace transform of $Z^Y_n$ in terms of $Q^Y(s)$.  A similar formula can be found in  \cite[Lemma 7.1]{ThZw}.

\begin{Lem}\label{Lem:decompose-Z}
Suppose that the condition \ref{Thm:Z-cond-1}  of Theorem $\ref{Thm:Z}$ is satisfied. Let $s_1>0$ and $s_2\geq0$.
 Then we have
\begin{align}
   &\int_Y 
   \bigg(\sum_{n\geq0} e^{-ns_1}\exp(-s_2 Z^Y_n)
   \bigg)
   \bigg(\sum_{n\geq0} e^{-n(s_1+s_2)}
   \hat{T}^n 1_{Y_n}
   \bigg)\,d\mu
   =
   \frac{Q^Y (s_1)}{1-e^{-(s_1+s_2)}}.
\label{Lem:decompose-Z-0}    	
\end{align}
\end{Lem}

\begin{proof}
Note that $Z^Y_n=Z^Y_{n-k}\circ T^k+k$ on $\{\varphi=k\}$ $(1\leq k \leq n)$ and $Z^Y_n=0$ on $\{\varphi>n\}$, and hence
\begin{equation}
   \exp(-s_2 Z^Y_n)
   =
   \begin{cases}
   \exp(-s_2 Z^Y_{n-k}\circ T^k) e^{-k s_2},
   &\text{on $\{\varphi=k\}$ $(1\leq k \leq n)$,}
   \\
   1,
   &\text{on $\{\varphi>n\}$.}	
   \end{cases}
\end{equation}
Therefore, for $n\in\bN_0$,
\begin{equation}
\begin{split}
	&\int_Y \bigg(e^{-ns_1}\exp(-s_2 Z^Y_n)\bigg)\,d\mu
	\\[5pt]
	&=\int_Y\bigg(\sum_{k=1}^n e^{-ns_1}
	\exp(-s_2 Z^Y_{n-k}\circ T^k)e^{-k s_2}1_{Y\cap \{\varphi=k\}}\bigg)
	\,d\mu
	+e^{-n s_1}\mu(Y\cap\{\varphi>n\})
	\\[5pt]
	&=
	\int_Y \sum_{k=1}^n \bigg(e^{-(n-k)s_1}
	\exp(-s_2 Z^Y_{n-k})\bigg)
	\bigg(e^{-k(s_1+s_2)}\hat{T}^k1_{Y\cap\{\varphi=k\}}\bigg)
	\,d\mu
	\\
	&\quad
	+e^{-ns_1}\mu(Y\cap\{\varphi>n\}).
\end{split}
\end{equation}
By taking the sum over $n\in\bN_{0}$, we get
\begin{equation}
\begin{split}
	&\int_Y \bigg( \sum_{n\geq0}e^{-ns_1}
	\exp(-s_2 Z^Y_n)\bigg)\,d\mu
	\\
	&=\int_Y
	\bigg(\sum_{n\geq0}e^{-ns_1}
	\exp(-s_2 Z^Y_n)\bigg)
	\bigg(\sum_{k\geq1}e^{-k(s_1+s_2)}
	\hat{T}^k1_{Y\cap\{\varphi=k\}}\bigg)\,d\mu
	+Q^Y(s_1),
\end{split}
\end{equation}
and hence
\begin{align}
\int_Y
	\bigg(\sum_{n\geq0}e^{-ns_1}\exp(-s_2 Z_n)\bigg)
	\bigg(1_Y-\sum_{k\geq1}e^{-k(s_1+s_2)}\hat{T}^k1_{Y\cap\{\varphi=k\}}\bigg)\,d\mu
	=
	Q^Y(s_1).
	\label{eq1:Lem:decompose-Z}	
\end{align}
As shown in \cite[(5.3)]{ThZw}, 
\begin{align}
	1_Y-\sum_{k\geq1}e^{-ks}\hat{T}^k1_{Y\cap\{\varphi=k\}}
	=
	(1-e^{-s})\sum_{n\geq0}e^{-ns}\hat{T}^n 1_{Y_n},
	\quad\text{a.e.\; $(s>0)$.}
	\label{eq2:Lem:decompose-Z}
\end{align}
Combining \eqref{eq1:Lem:decompose-Z} with \eqref{eq2:Lem:decompose-Z}, we obtain the desired result.
\end{proof}

\begin{Lem} \label{Lem2:decompose-Z}
Assume that the conditions \ref{Thm:Z-cond-1}, \ref{Thm:Z-cond-3}, \ref{Thm:Z-cond-4} of Theorem $\ref{Thm:Z}$  hold.
Let $q>0$ and let $\lambda:(0,\infty)\to(0,\infty)$ be a positive function with $\lambda(t)\to0$ $(t\to\infty)$.
 Then we have
\begin{equation}
	\int_0^\infty du\,e^{-qu}\int_Y d\mu_{H}\,\exp(-\lambda(t)Z^Y_{[ut]})
	\sim
	\frac{Q^Y (qt^{-1})}{(q+\lambda(t)t) Q^Y (qt^{-1}+\lambda(t))},
	\label{Lem:decompose-Z-1}  
\end{equation}
as $t\to\infty$.
\end{Lem}

\begin{proof}
By substituting $s_1=qt^{-1}$ and $s_2=\lambda(t)$ into \eqref{Lem:decompose-Z-0}, we have
\begin{align}
	&\int_Y 
   \bigg(\sum_{n\geq0} e^{-nqt^{-1}}\exp(-\lambda(t) Z^Y_n)
   \bigg)
   \bigg(\sum_{n\geq0} e^{-n(qt^{-1}+\lambda(t))}
   \hat{T}^n 1_{Y_n}
   \bigg)\,d\mu
   \\[5pt]
   &=\frac{Q^Y(qt^{-1})}{1-e^{-qt^{-1}-\lambda(t)}}
   \sim \frac{Q^Y(qt^{-1})}{qt^{-1}+\lambda(t)}
   \quad (t\to\infty).
   \label{eq3:Lem:decompose-Z} 
\end{align}
For $t>0$, set
\begin{align}
    H_t= \frac{1}{Q^Y(qt^{-1}+\lambda(t))}\sum_{n\geq0} e^{-n(qt^{-1}+\lambda(t))}
   \hat{T}^n 1_{Y_n}.
   \label{def:Ht-ZY}	
\end{align}
By the assumption and Lemma \ref{Lem:Ht}, there exists $t_0>0$ such that $\{H_t\}_{t\geq t_0}$ is uniformly sweeping for $1_Y$, and $H_t\to H$ $(t\to\infty)$ in $L^\infty(\mu)$.
We use Lemmas \ref{Lem:discrete}, \ref{Lem:IT} and \ref{Lem:bdd-Z} with $S_n=Z^Y_n$ and $R_{n,t}=\exp(-\lambda(t)Z^Y_n)$ to get
\begin{align}
	&\text{(the left-hand side of \eqref{eq3:Lem:decompose-Z})}
   \\[5pt]
	&\sim
	\bigg(
	\int_Y 
   \sum_{n\geq0} e^{-nqt^{-1}}\exp(-\lambda(t)Z^Y_n)
   \,d\mu_H\bigg)
   Q^Y(qt^{-1}+\lambda(t))
   \\[5pt]
   &
   \sim
   \bigg(t
   \int_0^\infty du\,e^{-qu}\int_Y d\mu_{H}\,\exp(-\lambda(t)Z^Y_{[ut]})\bigg)
    Q^Y(qt^{-1}+\lambda(t))
   \quad(t\to\infty).
   \label{eq4:Lem:decompose-Z}
\end{align}
Combining \eqref{eq3:Lem:decompose-Z} with \eqref{eq4:Lem:decompose-Z}, we obtain the desired result.
\end{proof}

We now prove Theorems \ref{Thm:Z} and \ref{Thm:Z-weak} by using Lemmas \ref{Lem:decompose-Z} and \ref{Lem2:decompose-Z}. We imitate the proof of \cite[Theorem 2]{KasYan}.

\begin{proof}[Proof of Theorem $\ref{Thm:Z}$]
Set $c(t)=c([t])$ for $t>0$.
Let $q,\lambda>0$ be positive constants. By substituting $\lambda(t)=\lambda/(c(t)t)$ into \eqref{Lem:decompose-Z-1}, we see that
\begin{align}
 &
 \int_0^\infty du \,e^{-qu}
 \int_X d\mu_H\,
     \exp\bigg(\!-\frac{\lambda Z^Y_{[ ut]}}{c(t)t} \bigg)    
 \sim
  \frac{c(t)Q^Y(qt^{-1})}
       {\lambda Q^Y(qt^{-1}+\lambda c(t)^{-1}t^{-1})}
   \quad
  (t\to\infty).	
   \label{eq:proof:Thm:Z-1}
\end{align}
By \eqref{eq:RV:Q}, \eqref{eq:c(n)} and the uniform convergence theorem for regular varying functions \cite[Theorem 1.5.2]{BGT}, we have $Q^Y(qt^{-1}+\lambda c(t)^{-1}t^{-1})\sim Q^Y(\lambda c(t)^{-1}t^{-1})$ $(t\to\infty)$ and
\begin{align}
	\frac{Q^Y(qt^{-1})}{Q^Y(qt^{-1}+\lambda c(t)^{-1}t^{-1})}
	&\sim
	\frac{Q^Y(qt^{-1})}{Q^Y(t^{-1})}\cdot
	\frac{Q^Y(t^{-1})}{Q^Y(\lambda c(t)^{-1}t^{-1})}
	\\
	&\sim
	\bigg(\frac{q c(t)}{\lambda}\bigg)^{-1+\alpha} \frac{\ell(t)}{\ell(c(t)t)}
	\quad(t\to\infty). 
	\label{eq:proof:Thm:Z-3}
\end{align}
Hence
\begin{align}
	&\frac{\ell(c(t)t)}{c(t)^\alpha\ell(t)}
	\int_0^\infty du \,e^{-qu}
 \int_X d\mu_H\,
     \exp\bigg(\!-\frac{\lambda Z^Y_{[ ut]}}{c(t)t} \bigg)
    \\
    &\to
     q^{-1+\alpha}\lambda^{-\alpha}
    =
    \frac{1}{\Gamma(1-\alpha)}
    \bigg(\int_0^\infty e^{-qu} u^{-\alpha}\,du\bigg)\lambda^{-\alpha}
    \quad
    (t\to\infty).
\end{align}
We use Lemma \ref{Lem:Laplace} to get, for $0<u<\infty$,
\begin{align}
	 &\frac{\ell(c(t)t)}{c(t)^\alpha\ell(t)} \int_X 
     \exp\bigg(\!-\frac{\lambda Z^Y_{[ut]}}{c(t)t} \bigg)\,d\mu_H
     \\[5pt]
     \notag
     &\to
     \frac{1}{\Gamma(1-\alpha)} u^{-\alpha}
       \lambda^{-\alpha} 
     =  \frac{\sin(\pi\alpha)}{\pi} u^{-\alpha}
         \int_0^\infty e^{-\lambda s} s^{-1+\alpha}\,ds
     \quad
     (t\to\infty).
\end{align}
Here we used Euler's reflection formula $\Gamma(\alpha)\Gamma(1-\alpha)=\pi/\sin(\pi\alpha)$.
By the extended continuity theorem for Laplace transforms of locally finite measures \cite[Chapter XIII.1, Theorem 2a]{Fel2}, for $0\leq s_0<\infty$,
\begin{align}
   &\frac{\ell(c(t)t)}{c(t)^\alpha\ell(t)}
    \mu_H\bigg(\frac{Z^Y_{[ut]}}{c(t)t} \leq s_0\bigg)
  \\[5pt]
   &\to
    \frac{\sin(\pi\alpha)}{\pi} u^{-\alpha}
    \int_0^{s_0} s^{-1+\alpha}\,ds
   =
   \frac{\sin(\pi\alpha)}{\pi\alpha}
   \bigg(\frac{s_0}{u}\bigg)^\alpha
   \quad
   (t\to\infty).
      \label{eq:proof:Thm:Z-5}
\end{align}
Therefore we substitute $t=n$ and $u=s_0=1$ into \eqref{eq:proof:Thm:Z-5} and then obtain
\begin{align}
	  \mu_H\bigg(\frac{Z^Y_n}{n} \leq c(n) \bigg)
  \sim
 \frac{\sin(\pi\alpha)}{\pi\alpha}\frac{c(n)^\alpha\ell(n)}{\ell(c(n)n)}
 \quad(n\to\infty),
\end{align}
which is the desired result.
\end{proof}

\begin{proof}[Proof of Theorem \textup{\ref{Thm:Z-weak}}]
Set $c(t)=c([t])$ for $t>0$. By Chebyshev's inequality,
\begin{align}
    \mu_G
    \bigg(\frac{Z^Y_{[t]}}{t}\leq c(t)\bigg)
    \leq
    e
    \int_X	
    \exp\bigg(- \frac{Z^Y_{[t]}}{c(t)t}\bigg)
    \,d\mu_G.
    \label{eq1:Thm:Z-weak}
\end{align}
For each $t>0$, the map $(0,\infty)\ni u\mapsto \int_X \exp(-Z_{[ut]}/(c(t)t))\,d\mu_G\in [0,\infty)$ is non-increasing. Hence we have
\begin{align}
   \int_X	
    \exp\bigg(- \frac{Z^Y_{[t]}}{c(t)t}\bigg)
    \,d\mu_G
    &\leq
    \int_0^1 du
    \int_X d\mu_G\,\exp\bigg(- \frac{Z^Y_{[ut]}}{c(t)t}\bigg)
    \\
    &\leq
    e
    \int_0^\infty du\,e^{-u}
    \int_X d\mu_G\,\exp\bigg(- \frac{Z^Y_{[ut]}}{c(t)t}\bigg)
    \\
    &\leq
    et^{-1}
    \sum_{n\geq0}
    e^{-nt^{-1}}
    \int_X
    \exp\bigg(- \frac{Z^Y_{n}}{c(t)t}\bigg)
    G\,d\mu.
    \label{eq2:Thm:Z-weak}	
\end{align}
Here we also used \eqref{Lem:DL-1-2}.
Define $H_t$ $(t>0)$ as in \eqref{def:Ht-ZY} with $q=1$.
By the assumption and Lemma \ref{Lem:Ht}, there exists $t_0>0$ such that $\{H_t\}_{t\geq t_0}$ is uniformly sweeping for $1_Y$, and hence for $G$.  
We use Lemma \ref{Lem:IT-weak} with $R_{n,t}=\exp(-Z_n/(c(t)t))$ to get
\begin{align}
    \sup_{t\geq t_0}
    \frac{\sum_{n\geq0}
    e^{-nt^{-1}}
    \int_X
    \exp(-Z^Y_n/(c(t)t))
    G\,d\mu}
    {\sum_{n\geq0}
    e^{-nt^{-1}}
    \int_Y
   \exp(-Z^Y_n/(c(t)t))
    H_t\,d\mu}
    <\infty.
    \label{eq3:Thm:Z-weak}	
\end{align}
By substituting $q=1$ and $\lambda(t)=c(t)^{-1}t^{-1}$ into \eqref{eq3:Lem:decompose-Z} and making a similar estimate as in \eqref{eq:proof:Thm:Z-3}, we see that
\begin{align}
	&t^{-1}\sum_{n\geq0}
    e^{-nt^{-1}}
    \int_X
    \exp\bigg(- \frac{Z^Y_{n}}{c(t)t}\bigg)
    H_t\,d\mu
    \\
    &=
    \frac{t^{-1}}{1-\exp(t^{-1}+c(t)^{-1}t^{-1})}
    \cdot
    \frac{Q^{Y}(t^{-1})}{Q^Y(t^{-1}+c(t)^{-1}t^{-1})}
    \sim
     \frac{ c(t)^\alpha\ell(t)}{\ell(c(t)t)}
     \quad(t\to\infty).	
     \label{eq4:Thm:Z-weak}
\end{align}
Combining \eqref{eq1:Thm:Z-weak} with \eqref{eq2:Thm:Z-weak}, \eqref{eq3:Thm:Z-weak}, \eqref{eq4:Thm:Z-weak}, we obtain
\begin{align}
   \limsup_{t\to\infty}
   \frac{\ell(c(t)t)}{c(t)^\alpha\ell(t)}
   \mu_G \bigg(\frac{Z^Y_{[t]}}{t}\leq c(t)\bigg)
   <\infty,	
\end{align}
as desired.
\end{proof}

\section{Proof of Theorem \ref{Thm:SY}}
\label{Sec:proof-SY}

Let us represent double Laplace transform of $S^Y_n$ in terms of $Q^Y(s)$. We also refer the reader to \cite[Lemma 5.1]{ThZw} for a similar formula.

\begin{Lem}\label{Lem:decompose-SY}
 Suppose that the condition \ref{Thm:SY-cond-1}  of Theorem $\ref{Thm:SY}$ is satisfied. Let $s_1>0$ and $s_2\geq0$.
 Then we have
\begin{align}
   &
   (1-e^{-s_2})\int_Y 
   \bigg(\sum_{n\geq0} e^{-ns_1}\exp(-s_2 S_n^Y)
   \bigg)\,d\mu
   \\
   &
   +
   (1-e^{-s_1})e^{-s_2}
   \int_Y 
   \bigg(\sum_{n\geq0} e^{-ns_1}\exp(-s_2 S_n^Y)
   \bigg)
   \bigg(\sum_{n\geq0} e^{-ns_1}
   \hat{T}^n 1_{Y_n}
   \bigg)\,d\mu
   \\
   &=
   Q^Y(s_1).
\label{Lem:decompose-SY-1}    	
\end{align}
\end{Lem}

\begin{proof}
It is easy to see that $S_n^Y=S_{n-k}^Y\circ T^k+1$ on $\{\varphi=k\}$ $(1\leq k \leq n)$ and $S_n^Y=0$ on $\{\varphi>n\}$, which implies
\begin{equation}
	\exp(-s_2 S_n^Y)
	=
	\begin{cases}
	\exp(-s_2 S_{n-k}^Y\circ T^k)e^{-s_2},
	&\text{on $\{\varphi=k\}$, $1\leq k \leq n$,}
	\\
	1,
	&\text{on $\{\varphi>n\}$.}	
	\end{cases}
\end{equation}
Thus, for $n\in\bN_0$,
\begin{align}
	&\int_Y \bigg(e^{-ns_1}\exp(-s_2 S_n^Y)\bigg)d\mu
	\\
	&=
	\int_Y \sum_{k=1}^n\bigg(e^{-ns_1}\exp(-s_2 S_{n-k}^Y\circ T^k)e^{-s_2} 1_{Y\cap \{\varphi=k\}}\bigg)\,d\mu
	+e^{-ns_1}\mu(Y\cap \{\varphi>n\})
	\\
	&=
	e^{-s_2}\int_Y \sum_{k=1}^n 
	\bigg(e^{-(n-k)s_1}\exp(-\lambda(t)S_{n-k}^Y)\bigg)
	\bigg(e^{-ks_2}\hat{T}^k1_{Y\cap\{\varphi=k\}}\bigg)\,d\mu
	\\
	&\quad
	+e^{-ns_1}\mu(Y\cap\{\varphi>n\}).
\end{align}
By taking the sum over $n\in\bN_0$, we get
\begin{align}
	&\int_Y \bigg(\sum_{n\geq0}e^{-ns_1}\exp(-s_2 S_n^Y)\bigg)\,d\mu
	\\
	&=
	e^{-s_2}\int_Y  \bigg(\sum_{n\geq0} e^{-ns_1}
	\exp(-s_2 S_{n}^Y)\bigg)
	\bigg(\sum_{k\geq0}e^{-ks_1}\hat{T}^k1_{Y\cap\{\varphi=k\}}\bigg)\,d\mu
	+Q^Y(s_1),
\end{align}
and hence
\begin{align}
	&(1-e^{-s_2})\int_Y \bigg(\sum_{n\geq0}e^{-ns_1}\exp(-s_2 S_n^Y)\bigg)\,d\mu
	\\
	&+
	e^{-s_2}\int_Y \bigg(\sum_{n\geq0} e^{-ns_1}
	\exp(-s_2 S_{n}^Y)\bigg)\bigg(1_Y-\sum_{k\geq0}e^{-ks_1}\hat{T}^k1_{Y\cap\{\varphi=k\}}\bigg)\,d\mu
	\\
	&=Q^Y(s_1).
\end{align}
By using \eqref{eq2:Lem:decompose-Z}, we obtain the desired result.
\end{proof}

\begin{Lem}\label{Lem2:decompose-SY}
Assume that the conditions \ref{Thm:SY-cond-1} and \ref{Thm:SY-cond-3} of Theorem $\ref{Thm:SY}$ hold.
Let $q>0$ and let $\lambda:(0,\infty)\to(0,\infty)$ be a positive function with
\begin{equation}\label{eq1:Thm:SY-DL}
	\lambda(t)\to0\quad
	\text{and}
	\quad
	\frac{Q^Y(qt^{-1})}{\lambda(t)t}\to 0
	\quad
	(t\to\infty).
\end{equation}
Then we have
\begin{equation}
	\int_0^\infty du\,e^{-qu}
	\int_Y d\mu\,	
	\exp(-\lambda(t)S_{[ut]}^Y)
	\sim
	\frac{Q^Y(qt^{-1})}{\lambda(t)t}
	\quad(t\to\infty).
	\label{Lem:decompose-SY-2}  
\end{equation}
\end{Lem}

\begin{proof}
By substituting $s_1=qt^{-1}$ and $s_2=\lambda(t)$ into \eqref{Lem:decompose-SY-1}, we have
\begin{align}
   &
   (1-e^{-\lambda(t)})\int_Y 
   \bigg(\sum_{n\geq0} e^{-nqt^{-1}}\exp(-\lambda(t) S_n^Y)
   \bigg)\,d\mu
   \\
   &
   +
   (1-e^{-qt^{-1}})e^{-\lambda(t)}
   \int_Y 
   \bigg(\sum_{n\geq0} e^{-nqt^{-1}}\exp(-\lambda(t) S_n^Y)
   \bigg)
   \bigg(\sum_{n\geq0} e^{-nqt^{-1}}
   \hat{T}^n 1_{Y_n}
   \bigg)\,d\mu
   \\
   &=
   Q^Y(qt^{-1}).
   \label{eq1:Lem:decompose-SY}	
\end{align}
Let us prove the second term in the left-hand side of 
\eqref{eq1:Lem:decompose-SY} is negligibly small as $t\to\infty$.
By the assumption, there exists $C\in(0,\infty)$ such that, for any $n\in\bN$, we have $\|\sum_{k=0}^{n-1}\hat{T}^k 1_{Y_k}\|_{L^\infty(\mu)}\leq Cw^Y_n$. Let $s>0$. By Fubini's theorem,
\begin{align}
	\sum_{n\geq0}e^{-ns}\hat{T}^n 1_{Y_n}
	&=
	\sum_{n\geq0}\bigg((1-e^{-s})\sum_{k\geq n} e^{-ks}\bigg)
	\hat{T}^n 1_{Y_n}
    =
    (1-e^{-s})\sum_{k\geq0}e^{-ks}\bigg(\sum_{n=0}^k \hat{T}^n 1_{Y_n}\bigg),
\end{align}
and hence
\begin{align}
	\bigg\|\sum_{n\geq0}e^{-ks}\hat{T}^n 1_{Y_n}\bigg\|_{L^\infty(\mu)}
	&\leq
	(1-e^{-s})\sum_{k\geq0}e^{-ks}\bigg\|\sum_{n=0}^{k}\hat{T}^n 1_{Y_n}\bigg\|_{L^\infty(\mu)}
	\\
	&\leq
	C(1-e^{-s})\sum_{k\geq0}e^{-ks}w^Y_{k+1}
	\\
	&=C\sum_{n\geq0}e^{-ns}(w^Y_{n+1}-w^Y_n)
	=CQ^Y(s).
	\label{eq:bdd-Q}	
\end{align}
By using \eqref{eq1:Thm:SY-DL} and \eqref{eq:bdd-Q},
\begin{align}
 0&\leq 
\frac{\text{(the second term in the left-hand side of \eqref{eq1:Lem:decompose-SY})}}
     {\text{(the first term in the left-hand side of \eqref{eq1:Lem:decompose-SY})}}
   \\[5pt]
   &\leq
   \frac{C(1-e^{-qt^{-1}})Q^Y(qt^{-1})}
        {e^{\lambda(t)}-1}
   \leq
   \frac{C
   qQ^Y(qt^{-1})}
   {\lambda(t)t}
   \to0
   \quad
   (t\to\infty).
    \label{eq3:Lem:decompose-SY}
\end{align}
On the other hand, Lemma \ref{Lem:discrete} yields
\begin{align}
	&\text{(the first term in the left-hand side of \eqref{eq1:Lem:decompose-SY})}
	\\
	&
	\sim
   \lambda(t)t\int_0^\infty du\,e^{-qu}\int_Y d\mu\,\exp(-\lambda(t)S_{[ut]}^Y)
   \quad(t\to\infty).
   \label{eq4:Lem:decompose-SY}
\end{align}
Combining \eqref{eq1:Lem:decompose-SY} with \eqref{eq3:Lem:decompose-SY} and \eqref{eq4:Lem:decompose-SY}, we obtain the desired result.
\end{proof}

We now prove Theorem \ref{Thm:SY} by using Lemma \ref{Lem2:decompose-SY}.

\begin{proof}[Proof of Theorem $\ref{Thm:SY}$]
Set $\tilde{c}(t)=\tilde{c}([t])$ for $t>0$.
	Let $q,\lambda>0$ be positive constants. By substituting $\lambda(t)=\lambda/(\tilde{c}(t)a(t))$ into \eqref{Lem:decompose-SY-2}, we have
\begin{align}
 &\frac{1}{\tilde c(t)}
  \int_0^\infty du \,e^{-qu}
 \int_Y d\mu\,
     \exp\bigg(\!-\frac{\lambda S^{Y}_{[ut]}}{\tilde{c}(t)a(t)}
       \bigg)
\sim
     \frac{1}{\Gamma(1+\alpha)}\frac{Q^Y(qt^{-1})}
     {\lambda Q^Y(t^{-1})}
     \quad(t\to\infty).
     \label{eq1:proof-Thm:SY}
\end{align}
By \eqref{eq:RV:Q}, 
\begin{align}   
     &\text{(the right-hand side of \eqref{eq1:proof-Thm:SY})}
     \\[5pt]
     &\to \frac{q^{-1+\alpha}\lambda^{-1}}{\Gamma(1+\alpha)}
     =
    \frac{\sin(\pi\alpha)}{\pi\alpha}
    \bigg(\int_0^\infty e^{-qu}u^{-\alpha}\,du\bigg) 
    \bigg(\int_0^\infty  e^{-\lambda s}\,ds\bigg)
     \quad(t\to\infty).	
\end{align}
Hence we use Lemma \ref{Lem:Laplace} to get, for $0<u<\infty$,
\begin{equation}
\frac{1}{\tilde{c}(t)}
\int_Y
\exp\bigg(\!-\frac{\lambda S^{Y}_{[ut]}}{\tilde{c}(t)a(t)}
       \bigg)\,d\mu
\to	
\frac{\sin(\pi\alpha)}{\pi\alpha}
u^{-\alpha}\int_0^\infty e^{-\lambda s}\, ds
\quad (t\to\infty).
\end{equation}
By the extended continuity theorem for Laplace transforms, for $0\leq s_0<\infty$,
\begin{equation}
	\frac{1}{\tilde{c}(t)}\mu_{1_Y}
	\bigg(\frac{S_{[ut]}^Y}{\tilde{c}(t)a(t)}\leq s_0\bigg)
	\to
	\frac{\sin(\pi\alpha)}{\pi\alpha}
	u^{-\alpha}
	\int_0^{s_0}ds=\frac{\sin(\pi\alpha)}{\pi\alpha}\frac{s_0}{u^\alpha}
	\quad
	(t\to\infty).
\end{equation}
By substituting $t=n$ and $u=s_0=1$, we obtain the desired result.
\end{proof}

\section{Proofs of Theorems \ref{Thm:SA} and \ref{Thm:SA-weak}}
\label{Sec:proof-SA}

We can also represent double Laplace transform of $S_n^{A_i}$ $(i=1,\dots,d)$ in terms of $Q^{Y,A_i}(s)$ $(i=1,\dots,d)$. We refer the reader to \cite[Lemma 6.1]{ThZw} and \cite[Proposition 5.1]{SerYan19} for similar formulae.

\begin{Lem}
Suppose the condition \ref{Thm:SA-cond-1}  of Theorem $\ref{Thm:SA}$ is satisfied. Let $s>0$ and $s_1,s_2,\dots,s_d\geq 0$. Then we have
\begin{align}
  &
   (1-e^{-s})\int_Y \bigg(
   \sum_{n\geq0}e^{-ns}\exp
	\bigg(-\sum_{j=1}^d s_j S_n^{A_j}\bigg)
   \bigg)\,d\mu
   \\[5pt]
   &+\sum_{i=1}^d
  (e^{s_i}-e^{-s})
   \int_Y 
   \bigg(
   \sum_{n\geq0}e^{-ns}\exp
	\bigg(-\sum_{j=1}^d s_j S_n^{A_j}\bigg)
   \bigg)
   \bigg(
   \sum_{n\geq1}e^{-n(s+s_i)}
   \hat{T}^n 1_{Y_n\cap A_i}
   \bigg)\,d\mu
   \\[5pt]
   &=
   \mu(Y)+
   \sum_{i=1}^d
   Q^{Y,A_i}(s+s_i).
\label{Lem:decompose-SA-1}    	
\end{align}
\end{Lem}

\begin{proof}
Set
\begin{align}
	R_{n}
	=\exp
	\bigg(-\sum_{i=1}^d s_i S_n^{A_i}\bigg),
	\quad
	n\in\bN_0. 
\end{align}
Note that, for $n\in\bN$,
\begin{equation}
   R_{n}
   =
   \begin{cases}
   	R_{n-1}\circ T,
   	&\text{on $\{\varphi=1\}=T^{-1}Y$,}
   	\\
   	(R_{n-k}\circ T^{k})e^{-(k-1)s_i},
   	&\text{on $(T^{-1}A_i)\cap \{\varphi=k\}$
   	  $(1\leq i \leq d\;\; \text{and}\;\; 2\leq k \leq n)$,}
   	 \\
   	e^{-ns_i},
   	&\text{on $(T^{-1}A_i)\cap \{\varphi> n\}$ $(1\leq i \leq d)$.}
   \end{cases}
\end{equation}
Hence $\int_Y R_{0}\,d\mu=\mu(Y)$ and, for $n\in\bN$,
\begin{align}
   \int_Y &e^{-ns}R_{n}\,d\mu
   \\
   =& \int_{Y}
    e^{-ns}R_{n-1}\:\hat{T}1_{Y\cap T^{-1}Y}\, d\mu
   \\
    &+
    e^{-s}
    \sum_{i=1}^d
    \int_Y \sum_{k=2}^{n}
    \bigg(e^{-(n-k)s}R_{n-k}\bigg)\:
    \bigg(
    e^{-(k-1)(s+s_i)}
    \hat{T}^{k}1_{Y\cap (T^{-1}A_i)\cap\{\varphi=k\}}
    \bigg)
    \,d\mu
   \\
    &+\sum_{i=1}^d 
    e^{-n(s+s_i)}\mu(Y\cap (T^{-1}A_i)\cap\{\varphi>n\}).
\end{align}
By taking the sum over $n\in\bN_0$, we get
\begin{align}
	\int_Y &\bigg(\sum_{n\geq0}e^{-ns}R_{n}\bigg)\,d\mu
	\\
	=\;
	& \mu(Y)+e^{-s}\int_Y \bigg(\sum_{n\geq0}e^{-ns}R_{n}\bigg)\hat{T}1_{Y\cap T^{-1}Y}\,d\mu
	\\
	&+
	e^{-s}
	\sum_{i=1}^d
    \int_Y 
    \bigg(\sum_{n\geq0}
    e^{-ns}R_{n}\bigg)\:
    \bigg(
    \sum_{k\geq2}
    e^{-(k-1)(s+s_i)}
    \hat{T}^{k}1_{Y\cap (T^{-1}A_i)\cap\{\varphi=k\}}
    \bigg)\,
    d\mu
   \\
    &+\sum_{i=1}^d Q^{Y,A_i}(s+s_i).
\end{align}
By \eqref{eq:1YnAi}, we have
\begin{align}
   1_Y&=\hat{T}1_{T^{-1}Y}
      =\hat{T}1_{Y\cap T^{-1}Y}
       +\sum_{i=1}^d\hat{T}1_{A_i\cap T^{-1}Y}
       \\
       &=
      \hat{T}1_{Y\cap T^{-1}Y}
      +\sum_{i=1}^d \sum_{k\geq2}\hat{T}^k 1_{Y\cap (T^{-1}A_i)\cap \{\varphi=k\}},
\end{align}
which implies
\begin{align}
	&(1-e^{-s})\int_Y \bigg(\sum_{n\geq0}e^{-ns}R_{n}\bigg)\,d\mu
	\\
   &+\sum_{i=1}^d e^{-s}
     \int_Y
     \bigg( \sum_{n\geq0}e^{-ns}R_{n}\bigg)
     \bigg(\sum_{k\geq2}(1-e^{-(k-1)(s+s_i)})
      \hat{T}^k 1_{Y\cap (T^{-1}A_i)\cap \{\varphi=k\}}
      \bigg)\,d\mu
   \\
   &
   =\mu(Y)+\sum_{i=1}^d Q^{Y,A_i}(s+s_i). 
   \label{eq1:Lem:decompose-SA}   
\end{align}
In addition, we use \eqref{eq:1YnAi} to get, for $t>0$,
\begin{align}
	\sum_{k\geq2}(1-e^{-(k-1)t})
      \hat{T}^k 1_{Y\cap (T^{-1}A_i)\cap \{\varphi=k\}}
    &
    =
   \sum_{k\geq2} \bigg((e^t-1)\sum_{n=1}^{k-1}e^{-nt}\bigg)\hat{T}^k 1_{Y\cap (T^{-1}A_i)\cap \{\varphi=k\}}
    \\
    &=
    (e^t-1)\sum_{n\geq1} e^{-nt}\sum_{k> n}\hat{T}^k 1_{Y\cap (T^{-1}A_i)\cap \{\varphi=k\}}
    \\
    &=
    (e^t-1)\sum_{n\geq1} e^{-nt}\:\hat{T}^{n}1_{Y_n\cap A_i}.
    \label{eq2:Lem:decompose-SA}  
\end{align}
Combining \eqref{eq1:Lem:decompose-SA} with \eqref{eq2:Lem:decompose-SA}, we complete the proof. 
\end{proof}

\begin{Lem}\label{Lem:decompose-SA}
Assume that the conditions \ref{Thm:SA-cond-1},  \ref{Thm:SA-cond-3} and \ref{Thm:SA-cond-4} of Theorem $\ref{Thm:SA}$ hold.
Let $q>0$ be a positive constant and let $\lambda_i:(0,\infty)\to(0,\infty)$ $(i=1,\dots,d-1)$ be  positive functions with
\begin{gather}\label{Lem:decompose-SA-2}
	\lambda_i(t)\to0,\quad
	\lambda_i(t)t\to\infty
	\quad(i=1,\dots,d-1),
	\\
	\frac{Q^{Y,A_d}(qt^{-1})}
	     {\sum_{i=1}^{d-1}Q^{Y,A_i}(qt^{-1}+\lambda_i(t))}
	     \to \infty,
	\quad
	\frac{Q^{Y,A_d}(qt^{-1})}
	     {\sum_{i=1}^{d-1}\lambda_i(t)t Q^{Y,A_i}(qt^{-1}+\lambda_i(t))}
	     \to 0,
	     \label{Lem:decompose-SA-3}
\end{gather}
as $t\to\infty$. 
Then 
\begin{align}
	&
   \sum_{i=1}^{d-1}
   \lambda_i(t)
   \int_Y 
   \bigg(
   \sum_{n\geq0}e^{-nqt^{-1}}
   \exp\bigg(-\sum_{j=1}^{d-1} \lambda_j(t) S^{A_j}_{n}\bigg)\bigg)
   \bigg(\sum_{n\geq1}e^{-n(qt^{-1}+\lambda_i(t))}
           \hat{T}^n 1_{Y_n\cap A_i}\bigg)
   \,d\mu
   \\
   &\sim
   Q^{Y,A_d}(qt^{-1})
   \quad(t\to\infty).
\end{align}

\end{Lem}

\begin{proof}
 Set
\begin{align}
	&R_{n,t}=\exp\bigg(-\sum_{j=1}^{d-1} \lambda_j(t) S^{A_j}_{n}\bigg)
	\quad(n\in\bN_0,\;t>0),
	\label{lambda_j(t) S^{A_j}_{n}}
   \\
   &
   H_t^{(i)}
   =\frac{\sum_{n\geq1}e^{-n(qt^{-1}+\lambda_i(t))}
           \hat{T}^n 1_{Y_n\cap A_i}}
         {Q^{Y,A_i}(qt^{-1}+\lambda_i(t))}
   \quad(t>0,\;i=1,\dots,d-1),
   \label{H_t^{(i)}}
   \\
   &H_t^{(d)}
   =
   \frac{\sum_{n\geq1}e^{-nqt^{-1}}
           \hat{T}^n 1_{Y_n\cap A_d}}
        {Q^{Y,A_d}(qt^{-1})}
        \quad(t>0).	
\end{align}
By the assumption and Lemma \ref{Lem:Ht}, there exists $t_0>0$ such that $\{H_t^{(i)}\}_{t\geq t_0,\;i=1,\dots,d-1}$ is uniformly sweeping for $1_Y$. In addition $\{H_t^{(d)}\}_{t>0}$ is $L^\infty(\mu)$-bounded, which follows from a similar calculation as in \eqref{eq:bdd-Q}.
   
By substituting $s=qt^{-1}$, $s_i=\lambda_i(t)$ $(i=1,\dots,d-1)$ and $s_d=0$ into \eqref{Lem:decompose-SA-1}, we have
\begin{align}
   &(1-e^{-qt^{-1}})\int_Y \bigg(
   \sum_{n\geq0}e^{-nqt^{-1}}R_{n,t}
   \bigg)\,d\mu
   \\
   &+
  \sum_{i=1}^{d-1}(e^{\lambda_i(t)}-e^{-qt^{-1}})
  Q^{Y,A_i}(qt^{-1}+\lambda_i(t))
   \int_Y 
   \bigg(
   \sum_{n\geq0}e^{-nqt^{-1}}R_{n,t}  \bigg)
   H_t^{(i)}\,d\mu
   \\
   &+(1-e^{-qt^{-1}})
   Q^{Y,A_d}(qt^{-1})
   \int_Y 
   \bigg(
   \sum_{n\geq0}e^{-nqt^{-1}}R_{n,t}
   \bigg)
   H_t^{(d)}\,d\mu
   \\
   &=\mu(Y)
   +
   \sum_{i=1}^{d-1} Q^{Y,A_i}(qt^{-1}+\lambda_i(t))+Q^{Y,A_d}(qt^{-1}).
   \label{eq3:Lem:decompose-SA}
\end{align}
Note that 
\begin{align}
	\text{(the right-hand side of \eqref{eq3:Lem:decompose-SA})}
	\sim
	Q^{Y,A_d}(qt^{-1})
	\quad(t\to\infty),
	\label{eq-right:Lem:decompose-SA}
\end{align}
since $Q^{Y,A_i}(s)\to\infty$ $(s\to 0+,\; i=1,\dots,d)$ and the assumption \eqref{Lem:decompose-SA-3}.

Let us prove the second term in the left-hand side of \eqref{eq3:Lem:decompose-SA} is the leading term, and the first and third terms in the left-hand side of \eqref{eq3:Lem:decompose-SA} are negligible as $t\to\infty$.
We use Lemmas \ref{Lem:IT-weak} and \ref{Lem:bdd-SA} with  $H_t=H^{(i)}_t$ and $G=1_Y$ to get
\begin{align}
    C=\limsup_{t\to\infty}
    \max_{1\leq i\leq d-1}
    \frac
    {\int_Y (\sum_{n\geq0}e^{-nqt^{-1}}R_{n,t})\,d\mu}
    {\int_Y (\sum_{n\geq0}e^{-nqt^{-1}}R_{n,t})H_t^{(i)}\,d\mu}
    <\infty.	
   \label{eq4:Lem:decompose-SA}
\end{align}
Hence
\begin{align}
0&\leq
	\limsup_{t\to\infty}
	\frac{\text{(the first term in the left-hand side of \eqref{eq3:Lem:decompose-SA})}}
	{\text{(the second term in the left-hand side of \eqref{eq3:Lem:decompose-SA})}}
	\\
   &\leq
   \limsup_{t\to\infty}
	\frac{qC}{\sum_{i=1}^{d-1}\lambda_i(t)t\, Q^{Y,A_i}(qt^{-1}+\lambda_i(t))}
	=0.
\label{eq6:Lem:decompose-SA}	
\end{align}
In addition, we use the assumption \eqref{Lem:decompose-SA-3} to see
\begin{align}
 0&\leq
	\limsup_{t\to\infty}
	\frac{\text{(the third term in the left-hand side of \eqref{eq3:Lem:decompose-SA})}}
	{\text{(the second term in the left-hand side of \eqref{eq3:Lem:decompose-SA})}}
	\\
   &\leq
   \bigg(\sup_{t>0}\|H_t^{(d)}\|_{L^\infty(\mu)}\bigg)\bigg(\limsup_{t\to\infty}
   \frac{qCQ^{Y,A_d}(qt^{-1})}{\sum_{i=1}^{d-1}\lambda_i(t)t\, Q^{Y,A_i}(qt^{-1}+\lambda_i(t))}\bigg)
   =0.	
\end{align}
Therefore we get
\begin{align}
  &\text{(the left-hand side of \eqref{eq3:Lem:decompose-SA})}
  \\
  &\sim
  \sum_{i=1}^{d-1}
  \lambda_i(t)
  Q^{Y,A_i}(qt^{-1}+\lambda_i(t))
  \int_Y \bigg(\sum_{n\geq0}e^{-nqt^{-1}}R_{n,t}\bigg)H_t^{(i)}\,d\mu
  \quad
  (t\to\infty).	
  \label{eq-left:Lem:decompose-SA}
\end{align}
Combining \eqref{eq3:Lem:decompose-SA} with \eqref{eq-right:Lem:decompose-SA} and \eqref{eq-left:Lem:decompose-SA}, we obtain the desired result.
\end{proof}

We now prove Theorems \ref{Thm:SA} and \ref{Thm:SA-weak} by using Lemma \ref{Lem:decompose-SA}.

\begin{proof}[Proof of Theorem $\ref{Thm:SA}$]
Set $c(t)=c([t])$ for $t>0$.
Let $q,\lambda,\lambda_1,\dots,\lambda_{d-1}>0$ and $\lambda_i(t)=\lambda\lambda_i/(c(t)t)$.  By \eqref{eq:c(n)}, \eqref{eq:RV:Qi} and the uniform convergence theorem for regular varying functions, we see
$Q^{Y,A_i}(qt^{-1}+\lambda_i(t))\sim Q^{Y,A_i}(\lambda_i(t))$ $(t\to\infty,\;i=1,\dots,d-1)$. By the Potter bounds for slowly varying functions, we see that $c(t)^{-1+\alpha}\ell(t)/\ell(c(t)t)\to\infty$ and $c(t)^{\alpha}\ell(t)/\ell(c(t)t)\to0$ $(t\to\infty)$. Thus, for $i=1,\dots,d-1$, 
\begin{align}
\frac{Q^{Y,A_d}(qt^{-1})}
	     {Q^{Y,A_i}(qt^{-1}+\lambda_i(t))}
	     &
	     \sim
	     \frac{Q^{Y,A_d}(qt^{-1})}{Q^{Y,A_i}(\lambda_i(t))}
	     \sim
	     \frac{\beta_d q^{-1+\alpha}}{\beta_i (\lambda\lambda_i)^{-1+\alpha}}
	     \frac{c(t)^{-1+\alpha}\ell(t)}{\ell(c(t)t)}
	      \to\infty
	      \quad(t\to\infty).	
\end{align}
and
\begin{align}
\frac{Q^{Y,A_d}(qt^{-1})}
	     {\lambda_i(t)t Q^{Y,A_i}(qt^{-1}+\lambda_i(t))}
&\sim
	     \frac{\beta_d q^{-1+\alpha}}{\beta_i (\lambda\lambda_i)^{\alpha}}
	     \frac{c(t)^\alpha\ell(t)}{\ell(c(t)t)}     
\to 0
	     \quad(t\to\infty).	
	   	      \label{eq1:proof-Thm:SA}     	
\end{align}
Therefore
\eqref{Lem:decompose-SA-2} and \eqref{Lem:decompose-SA-3} are fulfilled. Define $R_{n,t}$ $(n\in\bN_0,\;t>0)$ and $H_t^{(i)}$ $(t>0,\;i=1,\dots,d-1)$ as in \eqref{lambda_j(t) S^{A_j}_{n}} and \eqref{H_t^{(i)}}, respectively.
By Lemma \ref{Lem:decompose-SA},
\begin{align}
	&
    \sum_{i=1}^{d-1}
    \lambda_i(t)Q^{Y,A_i}(qt^{-1}+\lambda_i(t))
    \int_Y
   \bigg(
   \sum_{n\geq0}
   e^{-nqt^{-1}}
   R_{n,t}\bigg)
   H_t^{(i)}\,d\mu
   \\
   &\sim
   Q^{Y,A_d}(qt^{-1})
   \quad (t\to\infty).
    \label{eq2:proof-Thm:SA}
\end{align}
Set $\tilde{H}_t=\sum_{i=1}^{d-1}\beta_i\lambda_i^\alpha H_t^{(i)}$ $(t>0)$. Then we use \eqref{eq1:proof-Thm:SA} and \eqref{eq2:proof-Thm:SA} to get
\begin{align}
   &
   \frac{\ell(c(t)t)}{c(t)^\alpha \ell(t)}
    t^{-1}\int_Y
   \bigg(
   \sum_{n\geq0}
   e^{-nqt^{-1}}
   R_{n,t}\bigg)
   \tilde{H}_t\,d\mu
   \to
   \beta_d q^{-1+\alpha}\lambda^{-\alpha}
   \quad (t\to\infty).
    \label{eq3:proof-Thm:SA}
\end{align}
By the assumption and Lemma \ref{Lem:Ht}, there exists $t_0>0$ such that $\{\tilde{H}_t\}_{t\geq t_0}$ is uniformly sweeping for $1_Y$, and $\tilde{H}_t\to \tilde{H}$ in $L^\infty(\mu)$ $(t\to\infty)$. By Lemmas \ref{Lem:Laplace}, \ref{Lem:IT} and \ref{Lem:bdd-SA},
\begin{align}
   t^{-1}\int_Y\bigg(
   \sum_{n\geq0}
   e^{-nqt^{-1}}
  R_{n,t}\bigg)
  \tilde{H}_t\,d\mu
   &\sim
   t^{-1}\int_Y\bigg(
   \sum_{n\geq0}
   e^{-nqt^{-1}}
  R_{n,t}\bigg)
   \tilde{H}\,d\mu
   \\
   &\sim
  \int_0^\infty du\,e^{-qu}
   \int_Y
   d\mu_{\tilde{H}}\,
   R_{[ut],t}
   \quad(t\to\infty).
   \label{eq4:proof-Thm:SA}	
\end{align}
Thus we use similar arguments as in the proof of Theorem \ref{Thm:Z} to obtain
\begin{align}
   & \frac{\ell(c(t)t)}{c(t)^\alpha \ell(t)}
   	\mu_{\tilde{H}}\bigg(\frac{\sum_{j=1}^{d-1}\lambda_j S^{A_j}_{[ut]}}{c(t)t} \leq s_0\bigg)
  \to
   \beta_d \frac{\sin(\pi\alpha)}{\pi\alpha}
   \bigg(\frac{s_0}{u}\bigg)^\alpha
   \quad
   (t\to\infty,\;s_0,u>0).
   \label{eq5:proof-Thm:SA}
\end{align}
By substituting $t=n$ and $s_0=u=1$, we complete the proof.
\end{proof}

\begin{proof}[Proof of Theorem \textup{\ref{Thm:SA-weak}}]
Set $c(t)=c([t])$ $(t>0)$ and $\lambda_i(t)=\lambda_i/(c(t)t)$ $(t>0,\;i=1,\dots,d-1)$. Define $R_{n,t}$ $(n\in\bN_0,t>0)$ as in \eqref{lambda_j(t) S^{A_j}_{n}}. By Chebyshev's inequality,
\begin{align}
    \mu_G
    \bigg(\frac{\sum_{j=1}^{d-1}\lambda_j S^{A_j}_{[t]}}{t}\leq c(t)\bigg)
    \leq
    e
    \int_X	
    R_{[t],t}
    \,d\mu_G.
    \label{eq1:Thm:SA-weak}
\end{align}
For each $t>0$, the map $(0,\infty)\ni u\mapsto \int_X R_{[ut],t}\,d\mu_G\in [0,\infty)$ is non-increasing. Hence we have
\begin{align}
   \int_X	
     R_{[t],t}
    \,d\mu_G
    &\leq
    \int_0^1 du
    \int_X d\mu_G\, R_{[ut],t}
    \\
    &\leq
    e
    \int_0^\infty du\,e^{-u}
    \int_X d\mu_G\, R_{[ut],t}
    \\
    &\leq
    et^{-1}
    \sum_{n\geq0}
    e^{-nt^{-1}}
    \int_X
    R_{n,t}
    G\,d\mu.
    \label{eq2:Thm:SA-weak}	
\end{align}
Define $H_t^{(i)}$ $(t>0,\;i=1,\dots,d-1)$ by \eqref{H_t^{(i)}} with $q=1$, and set $\tilde{H}_t=\sum_{i=1}^{d-1}\beta_i\lambda_i^{\alpha} H_t^{(i)}$. Then Lemma \ref{Lem:Ht} implies that there exists $t_0>0$ such that $\{\tilde{H}_t\}_{t\geq t_0}$ is uniformly sweeping for $1_Y$, and hence for $G$. By Lemma \ref{Lem:IT-weak},
\begin{align}
	\sup_{t\geq t_0}
	\frac{\sum_{n\geq0}
    e^{-nt^{-1}}
    \int_X
    R_{n,t}G
    \,d\mu}
    {\sum_{n\geq0}
    e^{-nt^{-1}}
    \int_X
    R_{n,t}\tilde{H}_t
    \,d\mu}
    <\infty.
    \label{eq3:Thm:SA-weak}	
\end{align}
By substituting $q=\lambda=1$ into \eqref{eq3:proof-Thm:SA}, we get
\begin{align}
	&t^{-1}
	\sum_{n\geq0}
    e^{-nt^{-1}}
    \int_X
    R_{n,t}\tilde{H}_t
    \,d\mu
\sim
  \beta_d 
   \frac{c(t)^\alpha\ell(t)}{\ell(c(t)t)}
   \quad(t\to\infty). 
    \label{eq4:Thm:SA-weak}	
\end{align}
Combining \eqref{eq1:Thm:SA-weak} with 	\eqref{eq2:Thm:SA-weak}, \eqref{eq3:Thm:SA-weak} and \eqref{eq4:Thm:SA-weak}, we obtain the desired result.
\end{proof}

\section{Applications to Thaler's maps}\label{Sec:Thaler}

All of our abstract results in Section \ref{Sec:Main} can be applied to Thaler's maps \cite{Tha98, Tha02,SerYan19}, AFN maps \cite{ThZw},  Markov chains on multiray \cite{SerYan19} and random iterations of piecewise linear maps \cite{HataYano, NNTY}.
Theorems \ref{Thm:Z-weak}, \ref{Thm:SY} and \ref{Thm:SA-weak} can also be applied to ``unbalanced'' AFN maps, interval maps with critical points, and random walks driven by Gibbs--Markov maps \cite{Zwe07cpt}.   
For simplicity, we are going to focus only on Thaler's maps with two indifferent fixed points \cite{Tha02} in this section. 

\begin{Asmp}[Thaler's map]\label{Asmp:Thaler}
Suppose that the map $T:[0,1]\to[0,1]$ satisfies the following conditions:
\begin{enumerate}[label=\textup{(\roman*)}]
\item For some $c\in(0,1)$, the restrictions $T:[0,c)\to[0,1)$ and $T:(c,1]\to(0,1]$ are strictly increasing, onto and can be extended to $C^2$ maps $T_0:[0,c]\to[0,1]$ and $T_1:[c,1]\to[0,1]$ respectively.
\item $T_0'>1$ and $T_0''> 0$ on $(0,c]$, $T_1'>1$ and $T''_1<0$ on $[c,1)$ and $T'(0)=T'(1)=1$.
\item For some $p\in(1,\infty)$, $a\in(0,\infty)$ and some positive, measurable function $\ell^*:(0,\infty)\to(0,\infty)$ slowly varying at $0$ such that
\begin{align}
   Tx-x\sim a^{-p}(1-x-T(1-x))
   \sim
   x^{p+1}\ell^*(x)
   \quad(x\to0+).
\end{align}
\end{enumerate}
\end{Asmp}

In the following we always impose Assumption \ref{Asmp:Thaler}. Let us summarize the facts which are shown in  \cite{Tha80, Tha83, Zwe98, Zwe00, Tha02, ThZw}. After that we will explain applications of our abstract results to Thaler's maps. 

Let $f_i$ denote the inverse function of $T_i$ $(i=0,1)$.
Then $T$ admits an invariant density $h$ of the form
\begin{align}
	h(x)=h_0(x)\frac{x(1-x)}{(x-f_0(x))(f_1(x)-x)}
	\quad(x\in(0,1)),
\end{align}
where $h_0$ is continuous and positive on $[0,1]$. 
In addition $h$ has bounded variation on $[\varepsilon,1-\varepsilon]$ for any $\varepsilon\in(0,1/2)$.
Define the $\sigma$-finite measure $\mu$ as $d\mu(x)=h(x)\,dx$, $x\in[0,1]$. Then $\mu([0,\varepsilon])=\mu([1-\varepsilon,1])=\infty$ for any $\varepsilon\in(0,1)$, and $T$ is a CEMPT on the $\sigma$-finite measure space $([0,1],\cB([0,1]), \mu)$.

Let $\gamma\in(0,c)$ be a $2$-periodic point of $T$. Then $T\gamma\in(c,1)$.
Take $c_0\in(0,\gamma]$ and $c_1\in[T\gamma,1)$ arbitrarily, and set
\begin{align}
   A_0=[0, c_0),\quad
   Y=[c_0, c_1],\quad
   A_1=(c_1, 1].	
\end{align}
Then $\mu(Y)\in(0,\infty)$, $\mu(A_i)=\infty$ $(i=0,1)$ and $Y$ dynamically separates $A_0$ and $A_1$.
As shown in the proofs of \cite[Lemma 3]{Tha02} and \cite[Theorem 8.1]{ThZw}, 
there exist $\mu$-probability density functions $H^{(0)}, H^{(1)}:[0,1]\to [0,\infty)$ which are supported on $Y$, have bounded variations and satisfy
\begin{align}
	\lim_{n\to\infty}\frac{\hat{T}^n1_{Y\cap (T^{-1}A_i)\cap \{\varphi=n\} }}{\mu(Y\cap (T^{-1}A_i)\cap \{\varphi=n\})}
	=
	H^{(i)}
	\quad \text{in $L^\infty(\mu)$}\;\;(i=0,1).
\end{align}
Therefore we use \eqref{eq:1YnAi} and \eqref{eq:wYAi} to get
\begin{align}
	\lim_{n\to\infty}\bigg(\frac{1}{w_n^{Y,A_i}}
	 \sum_{k=0}^{n-1} \hat{T}^k 1_{Y_k\cap A_i}\bigg)
   = H^{(i)}
   \quad \text{in $L^\infty(\mu)$}
   \;\;(i=0,1).
   	\label{intermittent-Hi}
\end{align}
In addition,
\begin{align}
  &\mu(Y\cap (T^{-1}A_0)\cap \{\varphi=n\})
  \sim
  h(c)f_1'(0)(f_0^n(1)-f_0^{n+1}(1))
  \quad(n\to\infty),
  \\
  &\mu(Y\cap (T^{-1}A_1)\cap \{\varphi=n\})
  \sim
  h(c)f_0'(1)(f_1^{n+1}(0)-f_1^n(0))
  \quad
  (n\to\infty).	
\end{align}
Let
\begin{align}
   u_0(x)=\int_x^1 \frac{dy}{y-f_0(y)},
   \quad
   u_1(x)
   =
   \int_x^1 \frac{dy}{f_1(1-y)-(1-y)}
   \quad(x\in(0,1]).	
\end{align}
By \cite[Lemma 5]{Tha02} or \cite[Remark 1]{Zwe03},
\begin{align}
    f_0^n(1)\sim u_0^{-1}(n),
    \quad
    1-f_1^n(0)\sim u_1^{-1}(n)
    \quad(n\to \infty),	
\end{align}
which implies
\begin{align}
   w^{Y,A_0}_n
   \sim h(c)f_1'(0)\sum_{k=0}^{n-1}u_0^{-1}(k),
   \quad
   w^{Y,A_1}_n
   \sim h(c)f_0'(1)\sum_{k=0}^{n-1}u_1^{-1}(k)
   \quad(n\to\infty).	
\end{align}
Here we also used \eqref{eq:wYAi}.
Set
\begin{equation}
    \alpha=\frac{1}{p},
    \quad
    \beta_0=
    \frac{T'(c-)}{T'(c-)+T'(c+)a^{-1}}
    \quad
    \text{and}
    \quad
    \beta_1
    =
    1-\beta_0.	
\end{equation}
By the assumption and the basic properties of regular varying functions \cite[Theorems 1.5.11 and 1.5.12 and Corollary 1.7.3]{BGT}, there exists a positive function $\ell:(0,\infty)\to (0,\infty)$ slowly varying at $\infty$ such that
\begin{align}
   w_n^{Y,A_i}\sim
   \beta_i n^{1-\alpha}\ell(n)
   \quad(n\to\infty,\;i=0,1).
   \label{intermittent-wi}	
\end{align}
(For example, if $\ell^*(x)\sim C^*$ $(x\to0+)$ for some constant $C^*>0$, then $\ell(x)\sim C$ $(x\to\infty)$ for some constant $C>0$.)
By \eqref{intermittent-Hi} and \eqref{intermittent-wi}, we get
\begin{align}
  \lim_{n\to\infty}\bigg(\frac{1}{w_n^{Y}}
	 \sum_{k=0}^{n-1} \hat{T}^k 1_{Y_k}\bigg)
   = \beta_0H^{(0)}
     +\beta_1H^{(1)}
     =:H
     \quad \text{in $L^\infty(\mu)$},	
\end{align}
and
\begin{align}
w_n^{Y}\sim w_n^{Y,A_0}+w_n^{Y,A_1}\sim
   n^{1-\alpha}\ell(n)
   \quad(n\to\infty).	
\end{align}
Moreover, if $G:[0,1]\to [0,\infty)$ is Riemann integrable on $[0,1]$ with $\int_0^1 G(x)\,dx>0$, then $G$ is uniformly sweeping for $1_{[\varepsilon, 1-\varepsilon]}$ for any $\varepsilon\in(0,1/2)$, which  follows from \cite[Theorem 8.1]{ThZw}. 
Therefore $H, H^{(0)}, H^{(1)}$ are uniformly sweeping for $1_{[\varepsilon, 1-\varepsilon]}$ and hence for $1_Y$.
So we use our main results in Section \ref{Sec:Main} to obtain the following theorems.

\begin{Thm}
Let $\{c(n)\}_{n\geq 0}$ and $\{\tilde{c}(n)\}_{n\geq0}$ be positive-valued sequences satisfying \eqref{eq:c(n)} and \eqref{eq:tilde-c(n)}, respectively. Then we have \eqref{eq:Thm:Z}, \eqref{eq:Thm:SY} and
\begin{align}
   	 \mu_{H^{(i)}}\bigg(\frac{S_n^{A_i}}{n}\leq c(n)\bigg)
   \sim
   \frac{1-\beta_i}{\beta_i} \frac{\sin(\pi \alpha)}{\pi\alpha}
   \frac{c(n)^\alpha\ell(n)}{\ell(c(n)n)}
   \quad
   (n\to\infty,\;i=0,1).
\end{align}
\end{Thm}

\begin{Thm}\label{Thm:intermittent-bdd-below}
Assume $G\in\{u\in L^1(\mu)\::\:u\geq0\}$ admits a version which is Riemann integrable on $[0,1]$ with $\int_0^1 G(x)\,dx>0$. Then there exists some constant $C_1\in(0,\infty)$ such that, for any positive-valued sequences $\{c(n)\}_{n\geq 0}$ and $\{\tilde{c}(n)\}_{n\geq0}$ satisfying \eqref{eq:c(n)} and \eqref{eq:tilde-c(n)}, we have \eqref{eq:Cor2:Z-1}, \eqref{eq:Cor2:SY-1} and
\begin{align}
    C_1\leq \liminf_{n\to\infty}
    \frac{\ell(c(n)n)}{c(n)^\alpha\ell(n)}
    \mu_G\bigg(\frac{S_n^{A_i}}{n}\leq c(n)\bigg)
    \quad(i=0,1).	
\end{align}
\end{Thm}

\begin{Thm}\label{Thm:intermittent-bdd-above}
Assume $G\in\{u\in L^\infty(\mu)\::\:u\geq0 \}$ is supported on $[\varepsilon, 1-\varepsilon]$ for some $\varepsilon\in(0,1/2)$. Then there exists some constant $C_2\in(0,\infty)$ such that, for any positive-valued sequences $\{c(n)\}_{n\geq 0}$ and $\{\tilde{c}(n)\}_{n\geq0}$ satisfying \eqref{eq:c(n)} and \eqref{eq:tilde-c(n)}, we have \eqref{eq:Cor2:Z-2}, \eqref{eq:Cor2:SY-2} and
\begin{align}
    \limsup_{n\to\infty}
    \frac{\ell(c(n)n)}{c(n)^\alpha\ell(n)}
    \mu_G\bigg(\frac{S_n^{A_i}}{n}\leq c(n)\bigg)
    \leq C_2
    \quad(i=0,1).	
\end{align}	
\end{Thm}

\begin{Rem}
Let $\nu$ be a probability measure on $[0,1]$ which is supported on $[\varepsilon, 1-\varepsilon]$ for some $\varepsilon\in(0,1/2)$ and admits a Riemann integrable density $u$ with respect the Lebesgue measure. Then $G=u/h$ is also supported on $[\varepsilon, 1-\varepsilon]$ and Riemann integrable, and hence Theorems \ref{Thm:intermittent-bdd-below} and \ref{Thm:intermittent-bdd-above} can be applied to $\nu=\mu_{G}$.  	
\end{Rem}

\subsection*{Acknowledgements}

I am grateful to Professor Kouji Yano and Professor Yuko Yano for valuable discussions, and to Professor Alain Rouault for bringing the paper \cite{RYZ} to my attention.
This research was partially supported by JSPS KAKENHI Grant Number JP23K19010, and by the Research Institute for Mathematical Sciences, an International Joint Usage/Research Center located in Kyoto University.


\bibliographystyle{plain}

\end{document}